%% file: freek-laurent-final.tex
\documentclass[a4paper,12pt]{article}
\usepackage{latexsym}
\usepackage{amsmath}
\usepackage{amsthm}
\usepackage{amssymb}
\input{macro_L1.tex}

\input{macro_L2.tex}

\input{macro_E.tex}

\newcommand{\B}{{\bf B}}

\newcommand{\slt}{\sll(2,\Bbb C)}
\newcommand{\id}{\textrm{id}}
\newcommand{\inp}[1]{\left< #1\right>}

\newcommand{\p}[2]{\frac{\partial #1}{\partial #2}}
\newcommand{\pp}[1]{\frac{\partial}{\partial #1}}

\DeclareMathOperator{\boxtildeinv}{\widetilde{\square}^{-1}}
\DeclareMathOperator{\boxtilde}{\widetilde{\square}}
\DeclareMathOperator*{\Ker}{Ker}
\DeclareMathOperator*{\Span}{Span}
\DeclareMathOperator*{\sll}{\frak{sl}}

\renewcommand{\Im}{\operatorname{Im}}
\usepackage{color}

\newtheorem{theorem}{Theorem}[section]
\newtheorem{definition}[theorem]{Definition}
\newtheorem{proposition}[theorem]{Proposition}
\newtheorem{lemma}[theorem]{Lemma}

\newtheorem{corollary}[theorem]{Corollary}

\newtheorem{remark}[theorem]{Remark}

\title{Holomorphic normal form of nonlinear perturbations of nilpotent vector fields}
\author{Laurent Stolovitch \thanks{CNRS, Laboratoire J.-A. Dieudonn\'e
U.M.R. 6621, Universit\'e de Nice - Sophia Antipolis, Parc Valrose
06108 Nice Cedex 02, France. email : {\tt stolo@unice.fr} Research of L. Stolovitch was supported by ANR grant ``ANR-10-BLAN 0102'' for the project DynPDE} and  
Freek Verstringe \thanks{Royal Observatory of Belgium, Ringlaan 3,1180 Brussels, Belgium. email : {\tt freek.verstringe@oma.be}}}
\begin{document}

\maketitle

\begin{abstract} 
We consider germs of holomorphic vector fields at a fixed point having a nilpotent linear part at that point, in dimension $n\geq 3$. Based on Belitskii's work, we know that such a vector field is formally conjugate to a (formal) normal form. We give a condition on that normal form which ensure that the normalizing transformation is holomorphic at the fixed point. We shall show that this sufficient condition is a {\it nilpotent version} of Bruno's condition (A). In dimension 2, no condition is required since, according to  Str{\'o}{\.z}yna-{\.Z}o{\l}adek, each such germ is holomorphically conjugate to a Takens normal form. Our proof is based on Newton's method and $\frak{sl}_2(\Bbb C)$-representations.

\end{abstract}
\input{nilpotentnf5c}

\subsection{$\slt$-representations} 
\label{sec:machinery}

Throughout this section we work within the following setting.  We consider a 
representation $\{X,Y,H\}$ of $\slt$ acting on a finite dimensional vector space $\cV$, that is linear operators of $\cV$ and satisfy the relations \begin{align} [X,Y]=H,\qquad
[H,X]=2X,\qquad [H,Y]=-2Y.  \label{al:sl2relations} \end{align} Such a family $\{X,Y,H\}$ is also called an $\slt$-triple. The main definitions and propositions of this section can be found in \cite{Lie78-bourbaki, serre-semi}.

\begin{definition} 
Let $\{X,Y,H\}$ be an $\slt$-triple acting on a finite dimensional vector space $\cV$. A nonzero vector $b$ such that $Xb=0$ and $Hb=\lambda b$ is called a primitive vector. The corresponding eigenvalue $\la$ is called a weight.
\end{definition} 
We use the following lemma from $\slt$-representation theory: 
\begin{lemma}\cite{serre-semi}[IV-3] 
Let $\{X,Y,H\}$ be an $\slt$-triple acting on a vector space $\cV$. Let $b$ be a primitive vector. Let us define 
$$
e_m:=\frac{Y^mb}{m!},\quad  v_m:=Y^mb, \; m\geq 0
$$ 
and $e_{-1}:=v_{-1}:=0$.
The following properties hold:\\ 
\begin{minipage}[b]{0.5\linewidth} 
\begin{enumerate}
	\item $H(e_m)=\left( \lambda-2m \right)e_m$, 
	\item $Y(e_m)=\left( m+1\right)e_{m+1}$, \item $X(e_m)=\left( \lambda-m+1 \right)e_{m-1}$.  
\end{enumerate} 
\end{minipage}
\begin{minipage}[b]{0.5\linewidth} 
\begin{enumerate} 
	\item $H(v_m)=\left( \lambda-2m \right)v_m$, \item $Y(v_m)=v_{m+1}$,
	\item $X(v_m)=m\left( \lambda-m+1 \right)v_{m-1}$.  
\end{enumerate} 
\end{minipage} 
\label{lem:basicrepresentationlemma} 
\end{lemma}
\begin{proof}

Let $x$ be an arbitrary vector of $\cV$ such that $Hx=\lambda x$; then 
\begin{align*} 
HYx=[H,Y]x+YHx=-2Yx+\lambda Yx=\left(\lambda-2 \right)Yx.  
\end{align*} 
As a consequence we obtain $He_m=\left( \lambda-2m \right)e_{m}$ and $Y(e_m)=\left( m+1 \right)e_{m+1}$ follows from the definition
of $e_m$.  Finally we have that 
\begin{align*} 
mXe_m&=XYe_{m-1}=[X,Y]e_{m-1}+YXe_{m-1}\\ &=He_{m-1}+\left( \lambda-m+2\right)Ye_{m-2}\\ &=m\left( \lambda-m+1 \right)e_{m-1}.  
\end{align*}
 The formulas for $v_m$ are immediately deduced from the
formulas for $e_m$.  
\end{proof}
\begin{corollary}
If $\cV$ is a finite dimensional vector space then the weight $\lambda$ is an integer.  
\end{corollary}
\begin{corollary}\label{irred-mudule}\cite{serre-semi}[chap IV, theorem 2 and 3]
 Let $e$ be a primitive vector and consider the vector subspace $W=Span\{e,Ye,Y^2e,\ldots\}$ of $\cV$; then $W$ is an irreducible
 $\slt$-module. Every $\slt$-module can be decomposed into a direct sum of irreducible $\slt$-module. 
\end{corollary}

Let us also assume that the vector space $\cV$ is provided with a scalar product $\inp{.,.}$ and an associated norm $\norm{.}$. Assume that the operators $X,Y$ are adjoint w.r.t. the given scalar product~: for each $v,w\in \cV$, we have $\inp{Xv,w}=\inp{v,Yw}$, $\inp{Yv,w}=\inp{v,Xw}$.
\begin{lemma} The operator $H=XY-YX$ is self-adjoint w.r.t. the scalar product.  
\label{lem:selfadjoint} 
\end{lemma} 
\begin{lemma} 
\label{lem:representation_inproduct} 
Let $e_0$, $f_0$ be two distinct primitive vectors of $\cV$ such that
$\inp{e_0,f_0}=0$. Then, the following properties hold: 
\begin{enumerate} 
	\item $\inp{Y^m e_0,Y^n e_0}=0$ for all $m,n\in \N$, $m\neq n$.  
	\item $\inp{Y^m e_0,Y^n f_0}=0$ for all $m,n\in \N$.  
\end{enumerate} 
\end{lemma}
\begin{proof} We show the first property by induction. We have  
 \begin{align*} 
  \inp{Ye_0,e_0}=\inp{e_0,Xe_0}=0 
 \end{align*}
 Let $n>m$, then it follows from Lemma \ref{lem:basicrepresentationlemma} that 
 \begin{align*}
  \inp{Y^ne_0,Y^me_0}=\inp{Y^{n-1}e_0,(XY)Y^{m-1}e_0}=C_m\inp{Y^{n-1}e_0,Y^{m-1}e_0},
 \end{align*} 
 for a certain constant $C_m$ that depends on $m$. Continuing this way concludes the first property.  We proceed with the second
 property.  Remark first that for all $n\geq 0$ 
 \begin{align*} 
  \inp{Y^ne_0,f_0}=\inp{Y^{n-1}e_0,Xf_0}=0. 
 \end{align*} 
 Let now $n\geq m$
 \begin{align*} 
  \inp{Y^ne_0,Y^mf_0}&=\inp{Y^{n-1}e_0,(XY)Y^{m-1}f_0}\\
 &=C_m\inp{Y^{n-1}e_0,Y^{m-1}f_0}=\ldots=C_mC_{m-1}\ldots C_1\inp{e_0,Y^{m-n}f_0}=0. 
 \end{align*} 
\end{proof}

\begin{corollary} 
  \label{cor:irredcomp} 
  Let $\{X,Y,H\}$ be an $\slt$-triple acting on a finite dimensional Hilbert space $\cV$. Assume that the operators $X,Y$ are adjoint w.r.t. the given scalar product. Then, there exists an orthogonal decomposition into irreducible $\slt$-submodules of $\cV$.  
\end{corollary} 
\begin{proof} 
According to lemma \ref{irred-mudule}, $\cV$ can be written as a direct sum of irreducible submodule $W_i$'s. Each of them is generated by a primitive element $b^{(i)}$ which is an eigenvector of $H$ and belongs to the kernel of $X$. Let $\tilde V$ be the vector space spanned by these $b^{(i)}$'s.
Since $[H,X]=2X$, the kernel of $X$ is left invariant by $H$~: if $Xf=0$ then $-X(Hf)=0$. In particular, $\tilde V$ is invariant by $H$. The operator $H_{|\tilde V}$ is self-adjoint. Let $\{\tilde b^{(0)},\ldots,\tilde b^{(j)}\}$ be an orthonormal basis of eigenvectors of $H_{|\tilde V}$. Each $\tilde b^{(l)}$ is a primitive element and it generates an irreducible $\slt$-module $\tilde W_l:=\Span\{Y^m(\tilde b^{(l)})|\,m\in \N\}$. According to lemma \ref{lem:representation_inproduct}, the $\tilde W_l$'s are pairwise orthogonal. Hence, 
$\cV=\oplus W_i=\oplus^\bot\tilde W_i$.
\end{proof}

We explain now how the norm acts on the individual irreducible representations $\Span\{Y^l(b)|\,l\in \N\}$. We have the following
lemma.  
\begin{lemma} Suppose that $b$ is a primitive element of $\cV$ and $Hb=\lambda b$ then for each $m\geq 0$ we have 
\begin{equation}\label{lem:normlemma} 
  \norm{Y^mb}^2=\frac{m!\lambda!}{(\lambda-m)!}\norm{b}^2.
\end{equation}
\end{lemma} 
\begin{proof} According to lemma~\ref{lem:basicrepresentationlemma}, we have that 
$$
\norm{Y^m b}^2=\inp{Y^{m}b,Y^m b}=\inp{Y^{m-1}b,XY^mb}=m\left( \lambda-m+1 \right)\inp{Y^{m-1}b,Y^{m-1}b}. 
$$
It follows, by induction, that 
$\norm{Y^mb}^2=\frac{m!\lambda!}{(\lambda -m)!}\norm{b}^2$.
\end{proof}

Following the proof of lemma \ref{cor:irredcomp}, let $\{b^{(0)},\ldots, b^{(j)}\}$ be a basis of orthogonal eigenvectors of the operator $H_{|\Ker(X)}$.
Let us define $b_m^{(i)}:=Y^m(b^{i})/\norm{Y^m(b^{(i)})}$ for each $i\in \N$, $0\leq i\leq j$.  Using the same argumentation as in the
proof of corollary~\ref{cor:irredcomp} it follows that 
\begin{align} 
  \label{al:orthbasis} 
  \{b_m^{(i)}|\, m\in \N, 0\leq i\leq j\}\setminus\{0\} 
\end{align} 
is a {\bf basis of orthonormal eigenvectors of $H$}. We study the action of $X$, $Y$, $H$, $XY$ and $YX$
  on this normed basis, as this will be convenient for later use.  
\begin{lemma} 
  \label{lem:basislemma} 
Let $b\in \Ker(X)$ be an eigenvector of $H$ associated with eigenvalue $\lambda$. Let us set $b_m:=Y^m(b)/\norm{Y^m(b)}$, $m\geq 0$. Then the vectors $b_m$ satisfy the following
  properties: 
  \begin{align} 
	  \nonumber 
	  H(b_m)&=(\lambda-2m)b_m,\\ 
	  \nonumber 
	  Y(b_m)&=\sqrt{(m+1)(\lambda-m)}b_{m+1},\\ 
	  \nonumber
	  X(b_m)&=\sqrt{m(\lambda-m+1)}b_{m-1},\\ 
	  	  YX(b_m)&=m(\lambda-m+1)b_{m},\label{eigenvalue-box}\\ 
	  XY(b_m)&=(m+1)(\lambda-m)b_{m}.\nonumber
  \end{align} 
\end{lemma} 
\begin{proof} We use Lemma~\ref{lem:normlemma} and compute: 
\begin{align*}
	  H(b_m)&=\frac{H(Y^m(b))}{\norm{Y^m(b)}}=(\lambda-2m)b_m,\\
	  Y(b_m)&=\frac{Y(Y^m(b))}{\norm{Y^m(b)}}=\frac{Y^{m+1}(b)}{\norm{Y^m(b)}}=
	  \frac{b_{m+1}\norm{Y^{m+1}(b)}}{\norm{Y^m(b)}}=\sqrt{(m+1)(\lambda-m)}b_{m+1},\\
	  X(b_m)&=\frac{X(Y^m(b))}{\norm{Y^m(b)}}=m(\lambda-m+1)\frac{Y^{m-1}(b)}{\norm{Y^m(b)}}\\
		&= m(\lambda-m+1)\frac{b_{m-1}\norm{Y^{m-1}(b)}}{\norm{Y^m(b)}}=\sqrt{m(\lambda-m+1)}b_{m-1},\\
	  YX(b_m)&=\sqrt{m(\lambda-m+1)}Yb_{m-1}=m(\lambda-m+1)b_{m},\\
	  XY(b_m)&=\sqrt{(m+1)(\lambda-m)}Xb_{m+1}=(m+1)(\lambda-m)b_{m}.  
\end{align*} 
\end{proof}
It is a classical fact that $XY$ is selfadjoint positive semi-definite such that $\cV=\Ker(XY)\oplus\Im(XY)=\Ker(Y)$, $\Im(XY)=\Im(X)$, $\Ker(XY)=\Ker(Y)$ (the same holds for $YX$ if we interchange the role of $X$ and $Y$).

\section{The normal form procedure : The Newton method}
\label{sec:normal_form_procedure}

We consider a germ of holomorphic vector field at the origin of $\Bbb C^n$, $V=N+R_{\geq 2}$, where $N$ is its (nilpotent) linear part and $R_{\geq 2}=\sum_{k\geq 2}R_k$, with $R_k\in \mathcal{V}_k$, is a nonlinear perturbation of $N$. In this section, we investigate the conjugacy of $V$ to a normal form. We assume that $\{N^*, N, D':=[N^*,N]\}$ (and, as a consequence $\{d_0^*, d_0, H':=[d_0^*,d_0]\}$) is an $\slt$-triple acting on each space of homogeneous polynomials $\cP_k$ (resp. $\cV_k$) , $k\geq 2$. 
Hence, we can apply the entire abstract $\slt$-theory developed in Section \ref{sec:machinery}.

Let us assume that, at step $m$, we start with a a germ of holomorphic vector field at the origin of the form $V_m=N+f_mN^*+R_{m+1}$. 
In this formula, $NF_m:=N+f_mN^*$ is a polynomial normal form of degree $m$, $R_{m+1}:=B+C$ where $B\in \mathcal{V}_{m+1,2m}$, $C\in \mathcal{V}_{>2m}$ and $f_m$ is a polynomial of degree $m-1$ vanishing at the origin 
 and which is an invariant of both $N$ and $N^*$, that is $N(f_m)=N^*(f_m)=0$. 

At each step $m$ of the procedure, we try to find a suitable polynomial coordinate transformation $\Phi^{-1}=\id+U$, where
$U\in \mathcal{P}_{m+1,2m}$ such that the diffeomorphism $\Phi^{-1}=\id+U$ normalizes $V_m$ up to order $2m$. This means that $\Phi^{-1}$conjugates $V_m$ to $V_{2m}:=\Phi_*V_m=N+f_mN^*+\widetilde{B}+\widetilde{C}$ where $\widetilde{B}\in \mathcal{V}_{m+1,2m}$ is a normal form and $\widetilde{C}\in\mathcal{V}_{>2m}$. Let us set 
$R_{2m+1}:=\widetilde{C}$.
The conjugacy equation reads~:
\begin{align*}
\left( \id+DU \right)\left( N+f_mN^*+\widetilde{B}+\widetilde{C} \right)
=\left( N+f_mN^*+B+C\right)\circ \left( \id+U \right).
\end{align*}
The previous equation becomes~:
\begin{eqnarray}
NF_m+\widetilde{B}+\widetilde{C} +DU.(NF_m)+DU.(\widetilde{B}+\widetilde{C})&=& NF_m+D(NF_m).U+B\label{al:conjugacyeq1}\\
&&+\left( B\circ(\id+U)-B \right)+C\circ \left( \id+U \right)\nonumber\\
&&+\left\{ NF_m\circ \left( \id+U \right) -NF_m-D(NF_m).U\right\}.\nonumber
\end{eqnarray}
On the other hand, we also have
\begin{eqnarray}
\widetilde{C}&=&NF_m(\id+U)-NF_m+\left( B+C \right)\circ \left( \id+U \right)\nonumber\\
& &-\widetilde{B}-DU.(NF_m+\widetilde{B}+\widetilde{C}).\label{al:conjugacyeq3}
\end{eqnarray}
Hence,
\begin{eqnarray}
\widetilde{C}&=&\int_0^1 D(NF_m).(y+tU(y))U(y)\,dt+\left( B+C \right)\circ \left( \id+U \right)\nonumber\\
&&-\widetilde{B}-DU.(NF_m+\widetilde{B}+\widetilde{C}).\label{al:conjugacyeq4}
\end{eqnarray}
This will be used to estimate $\widetilde{C}$ within the Newton process.
We rewrite equation $(\ref{al:conjugacyeq1})$ under the following form~:
\begin{eqnarray}
\widetilde{B}-B+[NF_m,U]+\widetilde{C} & = & \left( B\circ(\id+U)-B \right)+C\circ \left( \id+U \right)\nonumber\\
&&+\left\{ NF_m \circ \left( \id+U \right) -NF_m  -DNF_m.U\right\}\nonumber\\
&&- DU\left(\widetilde{B}+\widetilde{C}\right).
\label{al:conjugacyeq2}
\end{eqnarray}

Since $U\in \mathcal{V}_{m+1,2m}$, the Taylor expansion at the origin of the right hand side of equation~$(\ref{al:conjugacyeq2})$ contains only terms of order strictly larger than $2m$ at the origin. Therefore, we have
$$
J^{2m}\left(\widetilde{B}-B+[NF_m,U]\right)=0,
$$
where $J^{2m}$ denotes the truncation at degree $2m$ of the Taylor expansion at the origin. We are led naturally to define the operator
\begin{eqnarray*}
d_{0,m+1,2m}: \cV_{m+1,2m} & \rightarrow & \cV_{m+1,2m}\\
U & \mapsto & J^{2m}\left([NF_m, U]\right).
\end{eqnarray*}

Let us set $Z:=\pi_{\Im(\square)}B$ 
and $U:=d_0^*(V)$ for some $V\in \cV_{m+1,2m}$. We want to solve the {\bf iterated cohomological equation}
\begin{align}
\pi_{\Im(\square)}(J^{2m}([N+f_mN^*,d_0^*(V)]))=Z=\pi_{\Im(\square)}J^{2m}(R_{m+1}).
\label{al:cohomeq_solvable}
\end{align}

The part that we do not remove from the conjugacy equation $(\ref{al:conjugacyeq2})$ is then given by
\begin{align}
\tilde B:=\pi_{\Ker(\square)}B-\pi_{\Ker(\square)}(J^{2m}([N+f_mN^*,d_0^*(V)])).
\label{al:cohomeq_solvable2}
\end{align}

By assumption, we have that $\tilde B=g_{2m}N^*$ where $g_{2m}\in \cO_N^N\cap\cO_N^{N^*}$ is polynomial of degree $\leq 2m-1$. We define $NF_{2m}:=NF_m+\tilde B$.

The main result of this section is the estimate of the cohomological equation 
\begin{prop}\label{prop-estim-cohom}
There exists a positive constant $\eta$, independent of $m$ such that if  $1/2\leq r\leq 1$ and $\left|D(NF^m-N)\right|_r+\left|NF^m-N\right|_r<\eta$, for any $Z\in \pi_{\Im(\square)}(\cV_{m+1,2m})$, then the unique solution $U=d_0^*V\in\cV_{m+1,2m}$ of the ``iterated cohomological equation'' $(\ref{al:cohomeq_solvable})$~: 
$$\pi_{\Im(\square)}(d_{0,m+1,2m}(U))=Z$$ satisfies
\begin{equation}
|U|_r \leq 2md|Z|_r.
\end{equation}
\end{prop} 
The next section is devoted to the proof of this proposition but we start with some preparations.
Recall that $\square=d_0d_0^*$ is a self-adjoint, positive semi-definite operator and that 
we have the decomposition $\mathcal{P}_{m+1,2m}=\Im(d_0)\oplus\Ker(d_0^*)=\Im(\square)\oplus\Ker(\square)$ together with $\Im(d_0)=\Im(\square)$ and $\Ker(d_0^*)=\Ker(\square)$.
As a consequence the operator
\begin{align*}
\boxtilde:\Im(\square)\rightarrow\Im(\square):V\mapsto \square(V),
\end{align*}
is an invertible operator. We define
\begin{align}
P_1~:&\,\mathcal{V}_{m+1,2m}\rightarrow\mathcal{V}_{m+1,2m};\qquad P_2~:\mathcal{V}_{m+1,2m}\rightarrow\mathcal{V}_{m+1,2m}\nonumber\\
&V\mapsto J^{2m}\left(f_md_0^*d_0^*(V)\right),\qquad V\mapsto J^{2m}\left(d_0^*(V)\left(f_m\right)N^*\right).\label{def-P}
\end{align}

We also define
\begin{align}
Q_1:=\boxtildeinv\pi_{\Im(\square)}P_1|_{\Im(\square)},\nonumber\\
Q_2:=\boxtildeinv\pi_{\Im(\square)}P_2|_{\Im(\square)}.\label{def-Q}
\end{align}
Hence, we have~:
\begin{align}
\pi_{\Im(\square)}J^{2m}([N+f_mN^*,d_0^*(V)] ) &=\pi_{\Im(\square)}\Big(\square(V)+J^{2m}(f_md_0^*d_0^*(V))-J^{2m}(d_0^*(V)\left(f_m\right)N^*)\Big)\nonumber\\
&=\square\pi_{\Im(\square)}\Big(\id+Q_1+Q_2\Big)(V)
\label{al:mainoperatorequation}
\end{align}
We consider the operator $\id +Q_1+Q_2$ as acting on $\Im(\Box)$. If it is invertible then 
the solution of equation $(\ref{al:cohomeq_solvable})$ is given by
\begin{equation}\label{equ-cohomo}
V=\left( \id +Q_1+Q_2 \right)^{-1}\boxtildeinv Z.
\end{equation}
We show in the next section that the operator $Q_1+Q_2$ is, in fact, nilpotent and we give a bound on this operator. This allows to invert and to estimate the solution of the equation since
it leads to:
\begin{align*}
\left( \id +Q_1+Q_2 \right)^{-1}=\sum_{l=0}^\alpha (Q_1+Q_2)^l.
\end{align*}
Although, the degree of nilpotency $\al$ depends on $m$,  \textbf{we essentially show, and this
is the key step, that $(\id+Q_1+Q_2)^{-1}$ can be bounded independently of the degree $m$.}
Such a sharp bound is needed to proof proposition \ref{prop-estim-cohom}.

\subsection{Estimate of the solution of the iterated cohomological equation}

This section is devoted to the proof of proposition \ref{prop-estim-cohom}.

We start by showing the nilpotency properties of several operators defined by $(\ref{def-P})$, $(\ref{def-Q})$ acting on $\cV_{m+1,2m}$.
\begin{lemma}
The operators $P_1,P_2,Q_1,Q_2, P_1+P_2,Q_1+Q_2$ are nilpotent. Moreover $P_1$ and $P_2$ commute pairwise.
\label{lem:nilpotencylemma}
\end{lemma}
\begin{proof}
Since $f_m$ vanishes at the origin, the multiplication by $f_m$ increases the order by $1$. Therefore, the operator $P_1$ is nilpotent. 
The operator $P_2$ is nilpotent too. In fact, since $f_m\in \cO^N_n\cap \cO^{N^*}_n$, we have 
\begin{align*}
P_2(P_2(V))=P_2(d_0^*(V)\left(f_m\right)N^*)&=J^{2m}\left(d_0^*(d_0^*(V)\left(f_m\right)N^*)\left(f_m\right)N^*\right)\\
&=J^{2m}\left(N^*\Big(d_0^*(V)\left(f_m\right)\Big) N^*\left(f_m\right)N^*\right)=0.
\end{align*}
To prove that $P_1$ and $P_2$ commute with each other is a little bit more intricate.
We first compute
\begin{align*}
  P_1P_2(V)&=J^{2m}\left(f_md_0^*d_0^*\left(d_0^*(V)( f_m )N^*\right)\right)=J^{2m}\left(f_mN^*N^*\left( d_0^*(V)( f_m ) \right)N^*\right)\\
&=J^{2m}\left(f_mN^*N^*N^*(V)\left( f_m \right) N^*\right).
\end{align*}
On the other hand, we have
\begin{align*}
P_2P_1(V)&=J^{2m}\left(d_0^*\left(f_md_0^*d_0^*(V)\right)\left(f_m\right)N^*\right)=J^{2m}\left(f_md_0^*d_0^*d_0^*(V)\left( f_m \right)N^*\right)\\
&=J^{2m}\left(f_mN^*d_0^*d_0^*(V)\left( f_m \right)N^*\right)=J^{2m}\left(f_mN^*(N^*(N^*(V)\left( f_m \right)))N^*\right).
\end{align*}
This proves that $P_1$ and $P_2$ are pairwise commuting. Since $P_1$ is nilpotent there exists a natural number $N>1$ (depending on $m$) such that $P_1^N=0$. Since $P_2^2=0$ then
$(P_1+P_2)^{N+1}=\sum_{k=1}^{N+1}\binom{N+1}{k} P_1^kP_2^{N+1-k}=0$.

Let us show that both $P_1$ and $P_2$ are upper triangular and $\square$ is diagonal with respect to a well-chosen basis. To this end it
suffices to write down the basis described in formula $(\ref{al:orthbasis})$ ordered by increasing degree. Since the operators
$P_1$ and $P_2$ increase the degree, they are automatically upper triangular in this basis. Furthermore it follows immediately
from Lemma~\ref{lem:basislemma} that $\square$ acts diagonally on this basis. One immediately deduces from these facts that the
operators $Q_1$, $Q_2$ and $Q_1+Q_2$ are upper triangular and hence nilpotent.
\end{proof}

We recall that $Y=d_0$, $X=d_0^*$, $H'=[d_0^*,d_0]=-H$ defines an $\slt$-triple acting on $\cV_{m+,2m}$.  
Following corollary \ref{cor:irredcomp}, we consider an orthonormal basis $\{b^{(1)},\ldots, b^{(k)}\}$ of primitive elements of $\cV_{m+1,2m}$ ($k$ depends on $m$) so that 
$$
\cV_{m+1,2m}=\bigoplus_{1\leq i\leq k}^{\bot}W_i
$$ 
is an orthogonal decomposition into irreducible $\slt$-submodules~: 
we have, for $1\leq i\leq k$, 
$$
W_i:=\text{Span}\left\{b_n^{(i)},\; n\in\Bbb N \right\}\quad \text{with }\; b_n^{(i)}=\frac{d_0^n(b^{(i)}_0)}{\norm{d_0^n(b^{(i)}_0)}}, n\in \N.
$$ 
The set 
\begin{equation}\label{basis}
\cB_{m+1,2m}:=\{b_n^{(i)}|\, 1\leq i\leq k,\, n\in \N\}\setminus \{0\}
\end{equation}
is an orthonormal basis of $\mathcal{V}_{m+1,2m}$ as immediately follows from Lemma~\ref{lem:representation_inproduct}.

The following lemma gives a bound for the operator $Q_1$.
\begin{lemma}
  The operator $Q_1$ is bounded by
  \begin{align}
	6\norm{f_m}.
	\label{al:formula}
  \end{align}
  
  \label{lem:boundofoperatorP1}
Furthermore, we have 
\begin{equation}\label{estim-Q1}
|Q_1(V)|_r\leq 6 |f_m|_r|V|_r.
\end{equation} 
\end{lemma}
\begin{proof}
We recall the definition of the operator 
\begin{align*}
Q_1(V)=\boxtildeinv\pi_{\Im(\square)}f_md_0^*d_0^*|_{\Im(\square)}(V).
\end{align*}
Since $\square=d_0d_0^*$ and since $f_m$ is an invariant of both $N$ and $N^*$, the multiplication operator by $f_m$ and the operator $\square$ commute pairwise. Therefore, we have
\begin{align*}
Q_1(V)=f_m\tilde Q_1(V),\quad\text{with }\; \tilde Q_1(V):=\boxtildeinv\pi_{\Im(\square)}d_0^*d_0^*|_{\Im(\square)}(V).
\end{align*}
We ``split'' the operator $\tilde Q_1$ into two parts: $d_0^*d_0^*$ and $\boxtildeinv$ and compute their actions using the decomposition into irreducible $\slt$-modules. In particular, we compute their actions on the basis $(\ref{basis})$. 
We use the theory developed in Section~\ref{sec:machinery} for that purpose.

According to lemma~\ref{lem:basislemma} and using the orthogonal basis $\cB_{m+1,2m}$, we obtain for each $j$~:
 \begin{align*}
   d_0^*d_0^*b_n^{(j)}&=X^2b_n^{(j)}=
   \sqrt{\left( \lambda-n+2 \right)\left( n-1 \right)}\sqrt{\left( \lambda-n+1\right)n}\,b_{n-2}^{(j)}.\\
\square b_n^{(j)}&=d_0d_0^*=YXb_n^{(j)}=n(\lambda-n+1)b_n^{(j)}. 
 \end{align*}
Here, $\lambda$ denotes the weight of the primitive element $b^{(j)}_0$. It follows that $Q_1(b_n^{(j)})=0$ for $0\leq n\leq 2$ and for $n\geq 3$ that:
\begin{align*}
  \tilde Q_1(b_n^{(j)})=
  \frac{\sqrt{\left( \lambda-n+2 \right)\left( n-1 \right)}\sqrt{\left( \lambda-n+1\right)n}}{(n-2)(\lambda-n+3)}b_{n-2}^{(j)}.
\end{align*}
On the other hand, we have 
$$
\frac{\sqrt{\left( \lambda-n+2 \right)\left( n-1 \right)}\sqrt{\left( \lambda-n+1\right)n}}{(n-2)(\lambda-n+3)}\leq 6.
$$
Let $V=\sum_{j=1}^k\sum_{n\geq 0}V_n^jb_n^{(j)}\in\mathcal{V}_{m+1,2m}^{(n)}$. According to the previous computations and estimates, we have
$$
\left\|\tilde Q_1(V)\right\|^2=  \left\|\sum_{j=1}^k\sum_{n\geq 0}V_n^j\tilde Q_1(b_n^{(j)})\right\|^2\leq \sum_{j=1}^k\sum_{n\geq 0}6^2|V_n^j|^2\|b_{n-2}^{(j)}\|^2\leq6^2\|V\|^2.
$$
Since $\tilde Q_1$ leaves each space of homogeneous vector fields invariant, we obtain as a consequence~:
\begin{eqnarray*}
\left|Q_1\left(V_{m+1}+\cdots+V_{2m}\right)\right|_r & \leq & |f_m|_r \left|\tilde Q_1\left(V_{m+1}+\cdots+V_{2m}\right)\right|_r\\
&\leq & |f_m|_r \sum_{k=m+1}^{2m}|\tilde Q_1(V_{k})|_r = |f_m|_r \sum_{k=m+1}^{2m}\|\tilde Q_1(V_{k})\|_kc_kr^k\\
&\leq & 6|f_m|_r \sum_{k=m+1}^{2m}\|V_{k}\|_kc_kr^k = 6|f_m|_r|V|_r. 
\end{eqnarray*}
This ends the proof of the lemma.
\end{proof}

We proceed with an estimate on the operator $Q_2$. 
Here the things tend to become a little more difficult since the operator $Q_2$ is not acting as a ``ladder'' operator as before. We can however make use of the following invariant subspace lemma~:
\begin{lemma}
The vector subspace of vector fields of the form $AN+BN^*+CH$, $A,B,C\in \mathcal{P}_{m+1,2m}$, is invariant under the action of $d_0,d_0^*$.
  \label{lem:invariantvectorspace}
\end{lemma}
\begin{proof}
  It is sufficient to show that this space is invariant by the action of $d_0$ and $d_0^*$. We recall that we consider the $\slt$-triple $\{N^*, N, H'\}$, with $H'=-H=[N^*, N]$.
  \begin{align}
	\nonumber
	[N,AN+BN^*+CH']&=N(A)N+N(B)N^*+N(C)H'+B[N,N^*]+C[N,H']\\
	&=\left( N(A)+2C \right)N+N(B)N^*+\left( N(C)-B \right)H'
	\label{al:invact1}
	\\
	\nonumber
	[N^*,AN+BN^*+CH']&=N^*(A)N+N^*(B)N^*+N^*(C)H'+A[N^*,N]+C[N^*,H']\\
	&=N^*(A)N+\left(N^*(B)-2C\right)N^*+\left( N^*(C)+A\right)H'.
	\label{al:invact2}
  \end{align}
\end{proof}

\begin{lemma}
Let $|||\nabla f_m|||=||\p{f_m}{x_1}||+\ldots+||\p{f_m}{x_n}||$.
  The operator $Q_2$ is bounded by $C_0|||\nabla f_m|||$, for some positive constant $C_0$, independent of $m$.
  \label{lem:boundofoperatorQ2}
Moreover, we have 
$$
|Q_2(V)|_r\leq C_0 \left(\left|\frac{\partial f_m}{\partial x_1}\right|_r+\cdots+\left|\frac{\partial f_m}{\partial x_n}\right|_r\right)\max(|N|_r,|N^*|_r,|H|_r)|V|_r
$$
\end{lemma}
\begin{proof}
  We recall the definition of the operator $Q_2$.
  \begin{align*}
Q_2(V)=\boxtildeinv\pi_{\Im(\square)}d_0^*(V)\left(f_m\right)N^*.
\end{align*}
Lemma~\ref{lem:invariantvectorspace} above shows that the space of vector fields of the form $AN+BN^*+CH'$ is invariant under the action of the Lie
algebra generated by $d_0$ and $d_0^*$. We compute the action of $\square$ on this space and deduce the action of $\boxtildeinv$
for elements of the form $B'N^*:=d_0^*(V)(f_m)N^*$. Of course, we shall obtain conditions on $B'$ in order to solve the equation. We then use this expression to compute a bound.
Using equations (\ref{al:invact1}) and (\ref{al:invact2}), we deduce that 
\begin{eqnarray}
  \square(AN+BN^*+CH')&=&\left( NN^*(A)+2\left( N^*(C)+A \right) \right)N  +\left( NN^*(B)-2N(C) \right)N^*\nonumber\\
  &&+\left( NN^*(C)+2C+N(A)-N^*(B) \right)H'\nonumber\\
  &=:& A'N+B'N^*+C'H'.\label{equ-box}
\end{eqnarray}
Let $D:=NN^*$ be the operator acting on polynomials. Let us set $K:=D+2I$. In matrix notation, operator $\square$ thus acts as
\begin{align}
\left(\begin{array}{ccc}
D+2I&0&2N^*\\
0&D&-2N\\
N&-N^*&D+2I
\end{array}\right)
\left(\begin{array}{c}
A\\B\\C
\end{array}\right)
=
\left(\begin{array}{c}
A^\prime\\B^\prime\\C^\prime
\end{array}\right).\label{syst-equ}
\end{align}
We want compute the inverse of that operator for a vector of the type $B^\prime N^*=[N^*,V](f_m)N^*$, so we set $A^\prime=C^\prime=0$. 
Therefore, we have
\begin{equation}
KA= -2N^*(C).
\label{eq:A}
\end{equation}
Composing the third line by $Y=N$ and adding it to the second line leads to
\begin{equation}
N\left(D-2NK^{-1}N^*\right)(C)=B'.
\label{eq:C}
\end{equation}
Finally, we have
\begin{equation}
D(B)=B'+2N(C).
\label{eq:B}
\end{equation}

We first decompose $V$ along the basis $\cB_{m+1,2m}$ as defined by $(\ref{basis})$~: $V=\sum_{i,n} V_n^{(i)}b_n^{(i)}$. Hence, we have $[N^*, V](f_m)=\sum_{i,n} V_n^{(i)}d_0^*(b_n^{(i)})(f_m)$. This motivates the following notation~:
For each irreducible submodule $W_i$ of $\mathcal V_{m+1,2m}$ associated to the index $1\leq i\leq k$, 
let $\lambda$ be its weight (we omit to write $i$), we shall set for $0\leq n\leq \la$~:
$$
a(n):=n(\la-n+1).
$$
We shall use the convention that $a(n)=0$ and $b_n^{(i)}=0$, for $n\leq 0$.
Let us also set 
$$
v_n:=d_0^*(b_n^{(i)})\left( f_m \right), \quad w_n:=b_n^{(i)}(f_m).
$$
Since $f_m$ is an invariant of both $N$ and $N^*$ then 
\begin{eqnarray}
N(w_n)&=&[N,b_n^{(i)}](f_n)= \sqrt{(n+1)(\lambda-n)}b_{n+1}^{(i)}(f_m)=\sqrt{(n+1)(\lambda-n)}w_{n+1},\label{n-wn}\\ N^*(w_n)&=&[N^*,b_n^{(i)}](f_n)=\sqrt{n(\lambda-n+1)}b_{n-1}^{(i)}(f_m)=\sqrt{n(\lambda-n+1)}w_{n-1}.\label{n*-wn}\\
H(w_n) &=& N^*N-NN^*(w_n)= N^*(\sqrt{a(n+1)}w_{n+1})-N(\sqrt{a(n)}w_{n-1})\nonumber\\
&=& (\lambda-2n)w_n.\label{weight-wn}
\end{eqnarray}
Since the elements $w_n$ have distinct weights, those which are non-zero are linearly independent \cite{serre-semi}[proposition 1, p.18]. Furthermore, we have
$$
v_n=[N^*,b_n^{(i)}](f_m)=N^*(b_n^{(i)}(f_m))=N^*(w_n)=\sqrt{a(n)}w_{n-1}.
$$
Since we want to invert on elements which are orthogonal to the kernel of $d_0^*$, it is sufficient to consider $B'N^*$ as linear combination of element of the form the form $d_0^*(b_n^{(i)})\left( f_m \right)N^*$, $n\geq 2$. Indeed, we have $d_0^*(b_0^{(i)})=0$ (by definition) and $d_0^*(d_0^*(b_1^{(i)})\left( f_m \right)N^*)=0$. Let us write 
$$
B'=\sum_{i}\sum_n B^{'(i)}_n b_n^{(i)}(f_m)=\sum_{i}\sum_n B^{'(i)}_n w_n. 
$$
Let us set $V=\sum_{i}\sum_{n\geq 1} V_n^{(i)}b_n^{(i)}\in {\mathcal V}_{m+1,2m}$. Since $B'=[N^*,V](f_m)$ then, according to $(\ref{n*-wn})$, we have
\begin{equation}
B^{'(i)}_n= V_{n+1}^{(i)}\sqrt{a(n+1)}.
\label{eq:bv}
\end{equation}
We recall that $K=D+2I$. The operators $D$ and thus $K$ are both diagonal when expressed in the basis $\{w_n\}$. Let us explain this in detail and let us compute their eigenvalues. 
Indeed, using Lemma~\ref{lem:basislemma} we have~:
$$
Kw_n=\left( NN^*+2I \right)b_n^{(i)}(f_m)=[\left( d_0d_0^*+2I \right)b_n^{(i)}]\left( f_m \right)=\left(a(n)+2 \right)w_n.
$$
It follows that
$$
  K^{-1}w_n=\frac{w_n}{\left( a(n)+2 \right)},\quad D w_{n-1}  =  a(n-1) w_{n-1}.
$$
As a consequence, we have
\begin{eqnarray}
K^{-1}N^*w_{n-1} & = & K^{-1}\sqrt{a(n-1)}w_{n-2}= \frac{\sqrt{a(n-1)}}{a(n-2)+2}w_{n-2},\nonumber\\
NK^{-1}N^*w_{n-1} & = & \frac{\sqrt{a(n-1)}}{a(n-2)+2}Nw_{n-2}=\frac{a(n-1)}{a(n-2)+2}w_{n-1},\nonumber\\
(D-2NK^{-1}N^*)w_{n-1} &= & \frac{a(n-1)a(n-2)}{a(n-2)+2}w_{n-1},\nonumber\\
N(D-2NK^{-1}N^*)w_{n-1} &= & \frac{\sqrt{a(n)}a(n-1)a(n-2)}{a(n-2)+2}w_{n}.\label{sol-basis}\\
\end{eqnarray}
Therefore, according to $(\ref{sol-basis})$, if 
\begin{equation}
B_0'^{(i)}=B_1'^{(i)}=B_2'^{(i)}=0,\label{restriction-b'}
\end{equation} then equation $(\ref{eq:C})$ has a unique solution 
\begin{equation}
C_{n-1}^{(i)} := \frac{a(n-2)+2}{\sqrt{a(n)}a(n-1)a(n-2)}B_{n}^{'(i)},\quad n\geq 3
\label{eq:Cn}
\end{equation}
with $C_0^{(i)}=C_1^{(i)}=0$.
Since we have
$$
K^{-1}N^*w_n= \frac{a(n)}{a(n-1)+2}w_{n-1},
$$
then equation  $(\ref{eq:A})$ has a unique solution given by 
\begin{equation}
A_{n-1}^{(i)} := -2\frac{a(n)}{a(n-1)+2}C_{n}^{(i)}=\frac{-2}{\sqrt{a(n+1)}a(n-1)}B_{n+1}^{'(i)},\quad n\geq 2
\label{eq:An}
\end{equation}
and  $A_0^{(i)}=0$. Finally, we have for $n\geq 3$, 
\begin{eqnarray*}
B_n^{'(i)}w_n+2C_{n-1}^{(i)}N(w_{n-1}) & = & \left(B_n^{'(i)}+2\sqrt{a(n)}C_{n-1}^{(i)}\right)w_{n}\\
&=& \left(1+\frac{2(a(n-2)+2)}{a(n-1)a(n-2)}\right)B_n^{'(i)}w_{n},\\
B_n^{'(i)}w_n+2C_{n-1}^{(i)}N(w_{n-1}) & = & 0,\quad n\leq 2,\\
D^{-1}(B_n^{'(i)}w_n+2C_{n-1}^{(i)}N(w_{n-1})) &= & \frac{1}{a(n)}\left(1+\frac{2(a(n-2)+2)}{a(n-1)a(n-2)}\right)B_n^{'(i)}w_{n}, \quad n\geq 3.
\end{eqnarray*}
Therefore, the unique solution of $(\ref{eq:B})$ such that $B_0^{(i)}= B_1^{(i)}=B_2^{(i)}=0$ is given by
\begin{equation}
B_n^{(i)}=\frac{1}{a(n)}\left(1+\frac{2(a(n-2)+2)}{a(n-1)a(n-2)}\right)B_n^{'(i)}, \quad n\geq 3.
\label{eq:Bn}
\end{equation}
Since $\la$ is a positive integer , we have
$$
\la\leq a(n)\leq \left(\frac{\la+1}{2}\right)^2,\quad 1\leq n\leq \la. 
$$
Therefore, the numbers 
\begin{eqnarray*}
\al_n & := & \frac{-2\sqrt{a(n+2)}}{\sqrt{a(n+1)}a(n-1)},\quad n\geq 2\\
\be_n & : = & \frac{\sqrt{a(n+1)}}{a(n)}\left(1+\frac{2(1+\frac{2}{a(n-2)})}{a(n-1)}\right),\quad n\geq 3\\
\ga_n & := &   \frac{\sqrt{a(n+1)}(1+\frac{2}{a(n-2)})}{\sqrt{a(n)}a(n-1)},\quad n\geq 3,
\end{eqnarray*}
 are uniformly bounded with respect to $n$ and to $\la$. Hence, $\sup_n|\alpha_n|$, $\sup_n|\beta_n|$, $\sup_n|\gamma_n|$ are uniformly bounded with respect to $m$ and $i$ (the label of irreducible submodules of $\cV_{m+1,2m}$).
According to equations $(\ref{eq:bv})$, $(\ref{eq:An})$, $(\ref{eq:Bn})$ and $(\ref{eq:Cn})$, we deduce that
\begin{align*}
A&=\sum_{i}\sum_{n\geq 2}\alpha_n V_{n+2}^{(i)}b_{n-1}^{(i)}(f_m)\\
B&=\sum_{i}\sum_{n\geq 3}\beta_n V_{n+1}^{(i)}b_{n}^{(i)}(f_m)\\
C&=\sum_{i}\sum_{n\geq 3}\ga_n V_{n+1}^{(i)}b_{n-1}^{(i)}(f_m)
\end{align*}
solve equation $(\ref{equ-box})$ with $A'=C'=0$ and $B'=d_0^*(V)(f_m)$ with restrictions $(\ref{restriction-b'})$.

Let $X=\sum_{j=1}^n X_j\frac{\partial }{\partial x_j}$ be a polynomial vector field. We have $\|X_j\|\leq \|X\|$ since $\|X\|=\sum_{j=1}^n\|X_j\|^2$. As a consequence, we have
$$
\|X(f)\|=\left\|\sum_{j=1}^nX_j\frac{\partial f}{\partial x_j}\right\|\leq \sum_{j=1}^n\|X_j\|\left\|\frac{\partial f}{\partial x_j}\right\|\leq \|X\| \sum_{j=1}^n\left\|\frac{\partial f}{\partial x_j}\right\|\leq \|X\||||\nabla f|||.
$$
Let $V\in\cV_{m+1,2m}$ be any unit vector field. It can be written along the orthonormal basis $\cB_{m+1,2m}$ as $V:=\sum_{(n,i)}V_{n,i}b_n^{(i)}$, where $\sum_{(n,i)}V_{n,i}^2=1$.
We compute
\begin{align*}
\|Q_2(V)\|\leq &
\left\|\sum_{i}\sum_{n\geq 2}V_{n+2,i}\alpha_n b_{n-1}^{(i)}(f_m)N\right\|+
\left\|\sum_{i}\sum_{n\geq 3}V_{n+1,i}\beta_n b_{n}^{(i)}(f_m)N^*\right\|\\
&+\left\|\sum_{i}\sum_{n\geq 3}V_{n+1,i}\de_n b_{n-1}^{(i)}(f_m)H\right\|\\
\leq& 
|||\nabla f_m
|||.||N||.\sup_n{\alpha_n}||V||+
|||\nabla f_m
|||.||N^*||.\sup_n{\beta_n}||V||\\
&+|||\nabla f_m
|||.||N^*||.\sup_n{\delta_n}||V||\\
\leq & C_0|||\nabla f_m
|||.
\end{align*}
Here $C_0$ is independent of $m$.
We have hence shown that $||Q_2||\leq C_0|||\nabla f_m
|||$.
Let us prove the second inequality. We first have 
$$
\left|\sum_{(n,i)}V_{n+2,i}\alpha_n b_{n-1}^{(i)}(f_m)N\right|_r \leq  \left|\sum_{(n,i)}V_{n+2,i}\alpha_nb_{n-1}^{(i)}(f_m)\right|_r|N|_r.
$$
Let us consider the polynomial vector field $X:=\sum_{(n,i)}V_{n+2,i}\alpha_nb_{n-1}^{(i)}$. Let us show that 
$$
|X(f_m)|_r\leq |X|_r\left(\left|\frac{\partial f_m}{\partial x_1}\right|_r+\cdots+\left|\frac{\partial f_m}{\partial x_n}\right|_r\right).
$$
Indeed, according to lemma \ref{prod-norm-r}, we have 
$$
|X(f_m)|_r\leq \max_j|X_j|_r\left(\left|\frac{\partial f_m}{\partial x_1}\right|_r+\cdots+\left|\frac{\partial f_m}{\partial x_n}\right|_r\right).
$$
We have $|X_j|_r\leq |X|_r$ since $\|\{X_j\}_k\|\leq \|\{X\}_k\|$ where $\{X_j\}_k$ denotes the homogeneous component of degree $k$ of $X_j$. As a consequence, we have
\begin{eqnarray*}
\left|\sum_{(n,i)}V_{n+2,i}\alpha_n b_{n-1}^{(i)}(f_m)N\right|_r &\leq &  \left|\sum_{(n,i)}V_{n+2,i}\alpha_nb_{n-1}^{(i)}(f_m)\right|_r|N|_r\\
  &\leq & \left|\sum_{(n,i)}V_{n+2,i}\alpha_nb_{n-1}^{(i)}\right|_r\left(\left|\frac{\partial f_m}{\partial x_1}\right|_r+\cdots+\left|\frac{\partial f_m}{\partial x_n}\right|_r\right)|N|_r.
\end{eqnarray*}
For $m+1\leq k\leq 2m$, let $I_k$ be the set of index $i$ for which $b_n^{(i)}$ is homogeneous of degree $k$ (for all $n$). Then, we have
$$
\left|\sum_{(n,i)}V_{n+2,i}\alpha_nb_{n-1}^{(i)}\right|_r= \sum_{k=m+1}^{2m}\sum_{i\in I_k}\left\|\sum_{n}V_{n+2,i}\alpha_nb_{n-1}^{(i)}\right\|c_kr^k.
$$
Since the $\cB_{m+1,2m}$ is an orthonormal basis, we have
$$
\left\|\sum_{n}V_{n+2,i}\alpha_nb_{n-1}^{(i)}\right\|^2\leq \sum_{n}V_{n+2,i}^2\alpha_n^2\leq \max_n\alpha_n^2\sum_{n}V_{n,i}^2.
$$
Therefore, we have
$$
\sum_{k=m+1}^{2m}\sum_{i\in I_k}\left\|\sum_{n}V_{n+2,i}\alpha_nb_{n-1}^{(i)}\right\|c_kr^k\leq \max_n\alpha_n\sum_{k=m+1}^{2m}\sum_{i\in I_k}\sqrt{\sum_{n}V_{n,i}^2}c_kr^k = \max_n\alpha_n|V|_r.
$$
We obtain similar estimates for the other terms and we are done.
\end{proof}

We recall the cohomological equation $(\ref{equ-cohomo})$
\begin{equation*}
V=\left( \id +Q_1+Q_2 \right)^{-1}\boxtildeinv Z
\end{equation*}
where $Z$ belongs to $\pi_{\Im(\square)}(\cV_{m+1,2m})$.
According the previous lemmas, we have
\begin{equation}\label{estim-cohomo}
|V|_r\leq \left[\sum_{k= 0}^{\alpha(m)}\left( 6 |f_m|_r+C_0 |\nabla f_m|_r\max(|N|_r,|N^*|_r,|H|_r)\right)^k\right]|\boxtildeinv Z|_r
\end{equation}
where we have written
$$
|\nabla f|_r:=\left(\left|\frac{\partial f_m}{\partial x_1}\right|_r+\cdots+\left|\frac{\partial f_m}{\partial x_n}\right|_r\right).
$$
The aim of the remainder of the proof is to obtain an upper bound for the norm of the right hand side of $(\ref{estim-cohomo})$. By definition, we have, 
\begin{eqnarray*}
NF^m -N & = & f_{m}N^*=\sum_{i=1}^n{f_m \left(\sum_{j=1}^n \tilde s_{i,j}x_{j}\right)\frac{\partial}{\partial x_i}}
\end{eqnarray*}
Let us set $L_i(x):= \sum_{i=j}^n \tilde s_{i,j}x_{j}$ for $1\leq i\leq n$. We write $f$ for $f_m$ and we decompose $f=\sum_{k=1}^{m-1}f^{(k)}$ as a sum of homogeneous polynomial $f^{(k)}$ of degree $k$ . We have
\begin{eqnarray*}
\|f L_i(x)\|^2 & = & \left< L_i(x)f,L_i(x)f\right > = \sum_{k=1}^m\left< L_i(x)f^{(k)},L_i(x)f^{(k)} \right>\\
& = & \sum_{k=1}^m\frac{1}{(k+1)!}\left<f^{(k)},\sum_{j=1}^n \overline{\tilde s_{i,j}}\frac{\partial (L_if^{(k)})}{\partial x_j}\right>_B \\
& = & \sum_{k=1}^m\frac{1}{(k+1)!}\left<f^{(k)},\sum_{j=1}^n |\tilde s_{i,j}|^2 f^{(k)}+ L_i(x)\sum_{j=1}^n\overline{\tilde s_{i,j}}\frac{\partial f^{(k)}}{\partial x_j}\right>_B. \\
\end{eqnarray*}
As a consequence, we have
$$
\|f L_i(x)\|^2 =  \sum_{k=1}^m\frac{1}{(k+1)!}\left(\left(\sum_{j=1}^n |\tilde s_{i,j}|^2\right)\|f^{(k)}\|_B^2+\|L_i^*(\partial)f^{(k)}\|_B^2\right).
$$
where $L_i^*(\partial)(f_k)=\sum_{j=1}^n\overline{\tilde s_{i,j}}\frac{\partial f_k}{\partial x_j}$.
Hence, we have 
\begin{eqnarray*}
|fL_i|_r&=&\sum_{k=1}^{m-1} \left< L_i(x)f^{(k)},L_i(x)f^{(k)} \right>^{1/2}c_{k+1}r^{k+1}\\
&=& \sum_{k=1}^{m-1}\left[\frac{1}{(k+1)}\left(\left(\sum_{j=1}^n |\tilde s_{i,j}|^2\right)\|f^{(k)}\|^2+\|L_i^*(\partial)f^{(k)}\|^2\right)\right]^{1/2}c_{k+1}r^{k+1}\\
&\geq & r\sum_{k=1}^{m-1}\frac{n^{1/2}}{\sqrt{k+1}}\left(\sum_{j=1}^N |\tilde s_{i,j}|^2\right)^{1/2}\|f^{(k)}\|c_{k}r^{k}\\
&\geq & r\left(\sum_{j=1}^n |\tilde s_{i,j}|^2\right)^{1/2} |f|_r.
\end{eqnarray*}
Since $r\geq 1/2$ and since there exists an $i$ for which $\sum_{j=1}^n |\tilde s_{i,j}|^2\neq 0$, we finally obtain
\begin{equation}\label{estim-f}
|f_m|_r\leq \frac{2}{\sum_{j=1}^n |\tilde s_{i,j}|^2} |NF^m-N|_r.
\end{equation}
We have $\frac{\partial f_m L_i}{\partial x_j} = f_m\frac{\partial L_i}{\partial x_j}+L_i(x)\frac{\partial f_m }{\partial x_j}$. According the previous estimate applied to $\frac{\partial f }{\partial x_j}$ instead of $f$, there exists an $i$ such that, if $r\geq 1/2$, then 
$$
\left|\frac{\partial f_m }{\partial x_j}\right|_r\leq \frac{2}{\sum_{j=1}^n |\tilde s_{i,j}|^2} \left|\frac{\partial f_m L_i}{\partial x_j}-f_m\tilde s_{i,j}\right|_r.
$$
As a consequence, there exists a positive constant $c$ such that, for all $1\leq j\leq n$ and all $m\geq 2$
\begin{equation}\label{estim-df}
\left|\frac{\partial f_m }{\partial x_j}\right|_r\leq c\left(\left|D(NF^m-N)\right|_r+\left|NF^m-N\right|_r\right)
\end{equation}

Let us choose $\eta>0$ such that 
\begin{equation}\label{eta}
(6 +C_0 nc\max(|N|_1,|N^*|_1,|H|_1))\eta<1/2.
\end{equation} 
Using estimates $(\ref{estim-f})$ and $(\ref{estim-df})$ into $(\ref{estim-cohomo})$, we obtain that if $1/2\leq r\leq 1$ and $\left|D(NF^m-N)\right|_r+\left|NF^m-N\right|_r<\eta$, then
$$
|V|_r\leq 2|\boxtildeinv Z|_r.
$$
According to Lemma $\ref{cauchy}$, since $V$ is a polynomial vector field of degree $\leq 2m$ and $1/2\leq r$, we have
\begin{equation}\label{estime-d_0^*}
\left|d_0^*V\right|_r\leq |V|_r|DN^*|_r+|DV|_r|N^*|_r\leq md|V|_r
\end{equation}
for some positive constant $d$, independent of $m$. On the other hand, according to $(\ref{eigenvalue-box})$ we have~: 
$$
|\boxtildeinv Z|_r\leq |Z|_r
$$
As a consequence, if  $1/2\leq r\leq 1$ and $\left|D(NF^m-N)\right|_r+\left|NF^m-N\right|_r<\eta$, then the unique solution $U=d_0^*V$ of the cohomological equation $(\ref{al:cohomeq_solvable})$ satisfies
\begin{equation}\label{estime-sol-cohom}
|U|_r \leq 2md|Z|_r
\end{equation}
and we are done.
\section{The iteration procedure towards convergence}

We adapt the proof by the first author in \cite{Stolo-ihes}[section 8]. Let $1/2<r\leq 1$ and $\eta>0$ be a positive number that is small enough so that condition $(\ref{eta})$ is satisfied.
For any integer $m\geq \lfloor 8n/\eta\rfloor +1$. We define
\begin{align*}
 \mathcal{NF}_m(r)&=\Big\{X\in \mathcal{V}_{>0}|\, \max(|X-N|_r,|D(X-N)|_r)<\eta-\frac{8n}{m}\Big\},\\
\mathcal{B}_{m+1}(r)&=\Big\{X\in \mathcal{V}_{>m}|\,|X|_r<1\Big\}.
\end{align*}
Let $m=2^k$ for some integer $k\geq 1$ and define
\begin{align*}
 \rho=m^{-2/m}r \text{ and } R=\gamma_k m^{-4/m}r, \text{ where }\gamma_k=(2md)^{-1/m}.
\end{align*}
Since $m^{1/m}\geq 1$ it is readily verified that $\rho<R<r\leq 1$ for $m$ large enough, say $m\geq m_0$.

Let us go back to the framework of beginning of section \ref{sec:normal_form_procedure}. Suppose that we have already normalized our vector field up to order $m$. Our starting point is a vector field of the form
$NF_m+R_{m+1}$, where $NF_m=N+f_mN^*$ is the polynomial part of degree $m$ of the normal form and $R_{m+1}$ is an analytic germ of order $\geq m+1$ at the origin.
The following proposition will play the role of one step in the Newton process:
\begin{proposition}\label{induction}
 Assume that $NF_m\in \mathcal{NF}_m(r)$, $R_{m+1}\in \mathcal{B}_{m+1}(r)$. Suppose that $m$ is large enough, say $m\geq m_0$, $m_0$ independent of $r$, then the unique $U\in d_0^*(\mathcal{V}_{m+1,2m})$ solution of $(\ref{al:cohomeq_solvable})$ is such that:
\begin{enumerate}
 \item $\Phi:=(\id+U)^{-1}$ is a diffeomorphism such that $|\id+U|_R< \rho$,
\item $\Phi_*(NF_m+R_{m+1})=N_{2m}+R_{2m+1}$ is normalized up to order $2m$,
\item $NF_{2m}\in \mathcal{NF}_{2m}(R)$, $ R_{2m+1}\in \mathcal{B}_{2m+1}(R)$.
\end{enumerate}
\end{proposition}
\begin{proof}
According to the properties of the norm used (see section \ref{section-norms}) and the definition of the new remainder $(\ref{al:conjugacyeq4})$, the proof of the first point, second point and of the inequality $R_{2m+1}\in \mathcal{B}_{2m+1}(R)$ of the proposition is identical to the proof of \cite{Stolo-ihes}[proposition 8.0.2]. Indeed, we obtained, from proposition \ref{prop-estim-cohom}, the estimate $|U|_r\leq \gamma_k^{-m}$.

We need to proof $NF_{2m}\in \mathcal{NF}_{2m}(R)$. Its proof slightly differs from the equivalent one in \cite{Stolo-ihes}[proposition 8.0.2]. We have to estimate the new normal form $NF^{2m}=NF^m+\tilde B$ as defined in $(\ref{al:cohomeq_solvable2})$. This expression contains a Lie bracket that we need to estimate.
According to $(\ref{al:cohomeq_solvable2})$, we have 
$$
\tilde B=\pi_{\Ker(\square)}B-\pi_{\Ker(\square)}(J^{2m}([f_mN^*,d_0^*(V)])).
$$
Indeed, since $\pi_{\Ker(\square)}(d_0d_0^*(V))=0$ (the range and kernel of $\square$ are supplementary spaces), then $\pi_{\Ker(\square)}(J^{2m}([N,d_0^*(V)]))=0$.
Hence, we have the following estimate~:
$$
|\tilde B|_R\leq |B|_R+|[f_mN^*,d_0^*(V)]|_R.
$$
In order to estimate the right hand side, we write
$$
[f_mN^*,d_0^*(V)]=f_md_0^*(d^*(V))-d_0^*(V)(f_m)N^*.
$$
This leads to the following estimate~:
$$
|[f_mN^*,d_0^*(V)]|_R\leq |f_m|_R|d_0^*(d^*(V))|_R+|d_0^*(V)(f_m)|_R|N^*|_R.
$$
According to $(\ref{estime-d_0^*})$ and $(\ref{estime-sol-cohom})$, and since $B$ (resp. $f_m$) has order $\geq m+1$ (resp. $\geq 1$) at the origin, we have
\begin{eqnarray*}
|[f_mN^*,d_0^*(V)]|_R &\leq & |f_m|_R2(md)^2|B|_R+2md|B|_R|\nabla f_m|_R|N^*|_R\\
&\leq & \left(\frac{R}{r}\right)^{m+2}\left(2(md)^2|f_m|_r|B|_r+2md|B|_r|\nabla f_m|_r|N^*|_r\right)\\
&\leq & \left(\frac{1}{2md}\right)^{1+2/{m}}\frac{1}{m^{4(1+2/m)}}\left(2(md)^2|f_m|_r|B|_r+2md|B|_r|\nabla f_m|_r|N^*|_r\right)
\end{eqnarray*}
According to the assumption of proposition \ref{prop-estim-cohom} and to the fact that $|B|_r\leq |B+C|_r\leq 1$, then for $1/2\leq r\leq 1$, we have
$$
|[f_mN^*,d_0^*(V)]|_R\leq C\eta\frac{1}{m^{3}}
$$
for some positive constant $C$, independent of $m$. Furthermore, by lemma \ref{cauchy}, we have
$$
|D([f_mN^*,d_0^*(V)])|_R\leq C'\eta\frac{1}{m^{2}}
$$
for some constant $C'$ independent of $m$. This shows that $NF_{2m}\in \mathcal{NF}_{2m}(R)$ and we are done.
\end{proof}

Let $1/2<r \leq 1$ be a positive number and let us consider the sequence $\{R_k\}_{k\geq 0}$ 
of positive numbers defined by induction as follow~:
\begin{eqnarray*}
R_0 & = & r\\
R_{k+1} & = & \gamma_km^{-2/m}R_k\quad\text{where}\quad m=2^k
\end{eqnarray*}
\begin{lemma}\label{small-div}
The sequence $\{R_k\}_{k\geq 0}$ converges and there exists an integer $m_1$ such that 
for all integer $k>m_1$, $R_k>\frac{R_{m_1}}{2}$.
\end{lemma}
\begin{proof}
We recall that $\gamma_k=\left(2md\right)^{-1/2^k}$.
We have $R_{k+1}=r\prod_{i=1}^k\gamma_i(2^i)^{-2^{1-i}}$, thus 
$$
\ln R_{k+1} = -\sum_{i=1}^k{\frac{\ln 2^{i+1}d}{2^i}}
-\ln c_1\sum_{i=1}^k{\frac{1}{2^i}}-2\ln 2\sum_{i=1}^k{\frac{i}{2^i}}
$$
all the sums are convergent as $k$ tends to infinity. 
It follows that there exists an integer $m_1$, such that 
$$
\prod_{i=m_1+1}^{+\infty}\gamma_i(2^i)^{-2^{1-i}}>1/2.
$$
Thus, if $k>m_1$ then $R_k=R_{m_1}\prod_{i=m_1+1}^k\gamma_i(2^i)^{-2^{1-i}}>\frac{R_{m_1}}{2}$.
\end{proof}

We are now in position to prove the theorem.

Let $N+R_{\geq 2}$ be a nonlinear holomorphic deformation of $N$ in a neighborhood of the origin 
in $\Bbb C^n$. We may assume that it is holomorphic in a neighborhood of the closed polydisc 
$D_1$.
Let $m_2=2^{k_0}$ be the smallest power of $2$ which is greater than $\max(m_0,2^{m_1})$ where $m_0$ is the integer 
defined in Proposition $\ref{induction}$.
By a polynomial change of coordinates, we can normalize $N+R_{\geq 2}$ up to order $m_2$ : 
in these coordinates,  $N+R_{\geq 2}$ can be written $NF^{m_2}+R_{m_2+1}$.
If necessary, we may apply a diffeomorphism $a\text{Id}$ with $a \in \Bbb C^*$ sufficiently small so that 
$(NF^{m_2},R_{m_2+1})\in {\cal NF}_{m_2}(1)\times {\cal B}_{m_2+1}(1)$. We define as above the sequence 
$\{R_k\}_{k\geq k_0}$, with $R_{k_0}=1$. Thus, for all integer $k>k_0$, we have $1/2<R_k\leq 1$.

Then following \cite{Stolo-ihes}[p.202-203], we have that for all $k\geq k_0$, that there exists a diffeomorphism $\Psi_k$ of $(\Bbb C^n,0)$ such that 
$\Psi_k^*(NF^{m_2}+R_{m_2+1})=: NF^{2^{k+1}}+R_{2^{k+1}+1}$ is normalized up to order $2^{k+1}$, 
$(NF^{2^{k+1}},R_{2^{k+1}+1})\in {\cal NF}_{2^{k+1}}(R_{k+1})\times {\cal B}_{2^{k+1}+1}(R_{k+1})$ and
$|\text{Id} -\Psi_k^{-1}|_{R_{k+1}}\leq \sum_{p=k_0}^k\frac{1}{2^{2p}}$.

Since $D(1/2)\subset D_{R_{k}}$ for all integers $k\geq k_0$, then the sequence $\{|\Psi_k^{-1}|_{1/2}\}_{k\geq k_0}$ is 
uniformly bounded. Moreover, the sequence $\{\Psi_k^{-1}\}_{k\geq k_0}$ converges coefficient wise to a formal 
diffeomorphism $\hat\Psi^{-1}$ (the inverse of the formal normalizing diffeomorphism). Therefore, this sequence converges 
in ${\cal H}_n^n(r)$ (for all $r<1/2$) to $\hat\Psi^{-1}$ (see \cite{Grauert-L1}). This means that the normalizing 
transformation is holomorphic in a neighborhood of $0\in \Bbb C^n$ and the theorem has been proven.

\newcommand{\etalchar}[1]{$^{#1}$}
\def\cprime{$'$} \def\cprime{$'$}

\end{document}

%% file: macro_L1.tex
\def\al{\alpha}
\def\be{\beta}
\def\ga{\gamma}

\def\de{\delta}

\def\la{\lambda}

\def\om{\omega}

\newcommand{\ti}{\tilde}

\newcommand{\C}{\mathbb{C}}

\newcommand{\N}{\mathbb{N}}

\def\cB{\mathcal{B}}
\def\cC{\mathcal{C}}

\newcommand{\cH}{\mathcal{H}}

\def\cL{\mathcal{L}}

\def\cO{\mathcal{O}}
\def\cP{\mathcal{P}}

\def\cV{\mathcal{V}}

\def\cc#1{_{(#1)}}

\def\norm#1{\lVert#1\rVert}

\def\om{\omega}



%% file: macro_L2.tex
\def\ftoday{le \space\number\day \space\ifcase\month\or
  janvier\or f\'evrier\or mars\or avril\or mai\or juin\or
  juillet\or ao\^ut\or septembre\or octobre\or novembre\or d\'ecembre\fi
  \space\number\year}

\def\real{I\kern-0.20em R}

\def\integer{I\kern-0.20em N}
\def\relative{{\rm \rlap Z\kern 2.2pt Z}}
\def\cc{\kern-.25em{\c c}}

\def\bc{\begin{center}}
\def\ec{\end{center}}
\def\=def{\stackrel{{\rm def}}{=}}

\newcounter{indconst}
\newcounter{auxconst}

\def\bit{\begin{itemize}}
\def\eit{\end{itemize}}
\def\ben{\begin{enumerate}}
\def\een{\end{enumerate}}
\def\bde{\begin{description}}
\def\ede{\end{description}}
 
\def\beq{\begin{equation}}
\def\eeq{\end{equation}}
\def\bfi{\begin{figure}[hbt] \begin{center}}
\def\efi{\end{center} \end{figure}}

\def\bce{\begin{center}}
\def\ece{\end{center}}

\newtheorem {theo} {Theorem}[section]

\newtheorem {prop}[theo]{Proposition}
\newtheorem {defi}[theo] {Definition}
 
\newtheorem {exam}[theo] {Example}

%% file: macro_E.tex
\renewcommand{\Im}[1]{ {\cal I}m  \left ( #1 \right )}
\renewcommand{\epsilon}{\varepsilon}

\newcommand{\Id}{{\rm Id}}

\newcommand{\ssi}[1]{si et seulement si}

\newcommand{\Norme}[2]%
  {\left |  #1 \right |_{\hspace{-0.9ex}\raisebox{-0.5ex}{ $\scriptstyle
{#2}$}}}
\newcommand{\NNorme}[2]%
{\left \|  #1 \right \|_{\hspace{-0.9ex}\raisebox{-0.5ex}{ $\scriptstyle
#2  $}}}

\newcommand{\norme}[2]%
  {\left |  #1 \right |_{\hspace{-0.9ex}\raisebox{-0.5ex}{ $\scriptscriptstyle
{#2}$}}}
\newcommand{\nnorme}[2]%
{\left \|  #1 \right \|_{\hspace{-0.9ex}\raisebox{-0.5ex}{ $\scriptscriptstyle
#2  $}}}

\newcommand{\sNorme}[2]%
  { |  #1 |_{\hspace{-0.9ex}\raisebox{-0.5ex}{ $\scriptstyle #2$}}}
\newcommand{\sNNorme}[2]%
{ \|  #1  \|_{\hspace{-0.9ex}\raisebox{-0.5ex}{ $\scriptstyle #2$}}}

\newcommand{\bNorme}[2]%
  { \Bigl |  #1 \Bigr |_{\hspace{-0.9ex}\raisebox{-0.5ex}{ $\scriptstyle #2$}}}
\newcommand{\bNNorme}[2]%
{ \Bigl \|  #1 \Bigr   \|_{\hspace{-0.9ex}\raisebox{-0.5ex}{ $\scriptstyle #2$}}}

\newcommand{\scalxx}[2]{\left \langle #1, #2 \right   \rangle}
\newcommand{\scalx}[3]%
{\scalxx{#1}{#2}_{\hspace{-0.4ex}\raisebox{-0.3ex}{$\scriptstyle {#3}$}}}

\newcommand{\Ucal}{{\cal U}}

%% file: nilpotentnf5c.tex
\section{Introduction}
In this article, we consider germs of holomorphic vector fields in a neighborhood of a fixed point in $\Bbb C^n$, $n\geq 2$. We are interested in the local classification under the action of the group of germs of biholomorphisms preserving the fixed point, which we may assume to be the origin. In the sequel, ``germ'' of vector field or map refers to a holomorphic germ at the origin.
The idea is to simplify, by a change of coordinates, the system of differential equations in order to better understand its dynamics. An elementary instance of such a problem is the following~: Let $A$ be a $n\times n$-matrix with complex coefficients. In order the understand the orbits $\{A^kz\}_{k\in \Bbb Z}$, $z$ near the fixed point $0$ of $A$, it is much easier to transform $A$ into its Jordan normal form $J$ by a change of coordinates $P$, $A=PJP^{-1}$. Then, one studies $\{J^ky\}_{k\in \Bbb Z}$, $Py=z$ and pulls back the information to the original problem. It was the idea of Poincar\'e to develop this point of view for vector fields. The difficulty is that the Lie algebra of germs of vector fields is infinite dimensional. The linear transformations are replaced by germs of diffeomorphisms fixing the origin. The special representant of the orbit of the action that plays the role of the Jordan normal form is called a normal form. A precise definition is given below.  This idea has been extensively developed by Arnold and his school, Bruno, Moser, Ecalle, Martinet-Ramis, Yoccoz, \ldots in the case where the linear part at the fixed point is a semi-simple matrix. 
The issue is that, although conjugacy to a formal normal form can always be obtained by a suitable formal transformation, it may not be possible to reach a normal form by a holomorphic transformation at the fixed point.

In this article, we investigate the normal form problem, in dimension $n\geq 3$, 
for holomorphic vector fields with a nilpotent linear part at the fixed point (assumed to be the origin). We give a sufficient condition that ensures that a germ of holomorphic vector field can be holomorphically conjugated to a normal form.

\subsection{Formal normal forms}

Following G. Belitskii \cite{Belitskii-formal, Belitskii-filtering}, E. Lombardi and L. Stolovitch defined in \cite{stolo-lombardi}, a notion of normal form of higher order perturbations of a given quasi-homogeneous vector field $S$ and the associated notion of ``generalized resonances'' (these are obstructions to find a formal power series transformation conjugating the perturbation of $S$ back to $S$ itself). Here, we shall first focus on the case where $S$ is linear, and then on the case where $S$ is nilpotent. 

Let $\cH_k$ be the space of homogeneous vector fields of degree $k$. 
For each natural number $k\geq 2$, we consider the ``cohomological
operator'' $d_0:=[S,.]:\cH_k\rightarrow \cH_{k}$ associated to $S$. Here, $[S,.]$ denotes the Lie bracket with $S$.
We define a space of normal forms of degree $k$ to be a supplementary space $\cC_k$ to $Im d_{0|\cH_k}$. We can show that there exists a formal transformation that conjugates each perturbation of $S$ of order $\geq 2$, $X:=S+R_{\geq2}$, to a formal vector field $NF:=S+\sum_{k\geq 2}v_k$ where $v_k\in \cC_k$ for each $k$. We call $NF$ a normal form (of ``style'' $\hat\cC=\oplus\cC_j$, see \cite{murdock-hyper}). Indeed, by induction on the the degree $k\geq 2$, assume that $X$ is normalized up to order $k-1$, that is $X=NF^{k-1}+R_{\geq k}$ and $NF^{k-1}\in \oplus_{j=2}^{k-1}\cC_j$. Let $\pi_{\text{Im }d_0}$ denote the projection onto the range of $d_0$ along $\cC_k$. Let us choose $U_k\in \cH_k$ to solve the {\bf cohomological equation} $d_0(U_k)=-\pi_{\text{Im }d_0}R_{k}$, where $R_k$ denotes the homogeneous polynomial of degree $k$ of $R_{\geq k}$. Then $(I-\pi_{\text{Im }d_0})R_{k}\in \cC_k$ and we have,  
$$
(\id+U_k)^{-1}_*X= NF_{k-1}+[S,U_k]+R_{k}+h.o.t.= NF_{k}+R_{>k},.
$$
where $NF_{k}:=NF_{k-1}+ (I-\pi_{\text{Im }d_0})R_{k}$. Here, $\Phi_*X:=D\Phi(\Phi^{-1})X(\Phi^{-1})$ denotes the conjugacy of $X$ by the diffeomorphism $\Phi$. It is classical that a normal form is not uniquely defined since one can add any $U_0\in \Ker d_{0|\cH_k}$ to the solution $U_k$ of $d_0(U_k)=-\pi_{\text{Im }d_0}R_{k}$. Therefore, defining for any k$\geq 2$, a supplementary space $\tilde V_k$ to $\Ker d_{0|\cH_k}$ in $\cH_k$, we can restrict the solution $U_k$ to belong to $\tilde V_k$, for all $k$, and refer to the normal form so obtained as {\it the} normal form of $X$.
The corresponding formal transformation $\cdots \circ(\id+U_k)\circ\cdots \circ(\id+U_1)$ is called {\it the associated normalizing} transformation.

Of course, we are interested in the case where these supplementary spaces can be computed effectively. One way to do so is as follows : $\cH_k$ is provided with an Hermitian scalar product (see (\ref{eq:inp})). Let us define 
$\cC_{k}$ (resp. $\tilde V_k$) to be the orthogonal complement to $\text{Im}(d_{0|\cH_k})$ (resp. to be equal to $\text{Im}(d_{0|\cH_k}^*)$), that is $\cH_{k}=\text{Im}(d_{0|\cH_k})\stackrel{\bot}{\oplus} \cC_{k}$. The space $\cC_{k}$ is the kernel of the adjoint $d_{0|\cH_{k}}^*$ of $d_{0|\cH_k}$ with respect to the Hermitian scalar product. 

When $S$ is a semi-simple linear vector field with eigenvalues $\lambda_1,\ldots, \lambda_n$, the space of normal forms is defined in terms of resonance relations $\sum_jq_j\lambda_j=\lambda_i$: the space $\cC_k$ generated by monomial vector fields $x^Q\partial_{x_i}$ with $Q=(q_1,\ldots,q_n)\in \N^n$, $|Q|=q_1+\cdots+q_n=k\geq 2$, $1\leq i\leq n$ which satisfies a resonance relation. We can also choose $\tilde V_k=\cC_k$ and obtain the classical Poincar\'e-Dulac normal form \cite{Arn2}.


When $S$ is nilpotent, the normal form spaces are more diverse and there are different ways of defining normal forms. In dimension $n=2$, F. Takens \cite{kn:Tak2} has shown that the space of normal forms associated to $S=y\partial_x$, $\cC_k$, can be chosen to be the vector space generated by $x^k\partial_x$ and  $x^k\partial_y$, $k\geq 2$. This has been generalized recently by E. Str{\'o}{\.z}yna and H. {\.Z}o{\l}adek \cite{zoladek-nil-multidim} to the case $n\geq 2$. 

There is another classical way to define normal forms of perturbations of nilpotent linear vector fields, by considering an $\slt$-triple associated to $S$ and their representations (see \cite{cushman-sanders, murdock-book}). First of all, it is classical (Jacobson-Morozov theorem)\cite{serre-semi, Lie78-bourbaki} that there exists linear vector fields $N:=S$, $M$ and $H$ of $\Bbb C^n$ such that $[N,M]=H$, $[H,N]=2N$, $[H,M]=-2M$.
A representation of $\slt$ in a finite dimensional vector space $V$, is a triple of endomorphisms ${X,Y,Z}$ such that
$$
[X,Y]=Z,\qquad
[Z,X]=2X,\qquad [Z,Y]=-2Y.
$$
Here, $[X,Y]$ denotes the bracket of endomorphism $XY-YX$. Such a family $\{X,Y,Z\}$ is called an $\slt$-triple. It is classical that $V= {Im X}\oplus Ker Y$. We apply this to $X=[N, .]$, $Y=[M,.]$ acting on the space of homogeneous polynomial vector fields~:
The supplementary space to the image of $d_0$ is then defined to be the kernel of $\text{ad}_M:=[M,.]$ restricted to the space of homogeneous polynomial vector fields. 

There exist also more involved constructions of formal normal forms 
dealing with uniqueness problems such as in \cite{kokubu} for perturbation 
of (quasi-)homogeneous vector fields or in \cite{sanders-spectral,gazor-spectral}, based on spectral sequences.

\subsection{Analytic normal form problem}

As we have seen, for each higher order perturbation of $S$, there exists a formal transformation (i.e a formal change of coordinates) to a formal normal form. We are interested in the situation where not only the perturbation is analytic in a neighborhood of the origin but also the normalizing transformation to its normal form (which is thus analytic too). 

Hence we are led to solve and estimate iteratively the solution $U_k$ 
of the ``cohomological equation'' $d_0(U_k)=F_k$ for a given polynomial $F_k$ of degree $k$ in the range of $d_0$. 

If $S$ is semi-simple, say $S=\sum_{i=1}^n\lambda_i x_i\frac{\partial}{\partial x_i}$, 
this leads to the so-called {\it small divisor problem} \cite{siegel-moser-book, moser-princeton, Arn2}~: the norm of $U_k$ is bounded by the norm of $F_k$ divided by smallest of the non-zero numbers of the form $|q_1\lambda_1+\cdots+ q_n\lambda_n-\la_j|$, $Q=(q_1,\ldots, q_n)\in \Bbb N^n$, $|Q|:=q_1+\cdots +q_n=k$. These numbers may accumulate to the origin when the degree $k$ tends to $+\infty$ causing a blow up of the normalizing transformation. These numbers are called ``small divisors'' (although they may not be small). The faster these numbers accumulate to zero, the smaller the radius of convergence of the formal transformation is. For instance, C. L. Siegel proved \cite{Siegel}, that if there exists $\tau\geq 0$ such that $|q_1\lambda_1+\cdots+ q_n\lambda_n-\la_j|> \frac{C}{|Q|^{\tau}}$ for all $Q\in \Bbb N^n$, $|Q|\geq 2$, then any higher order analytic perturbation of $S$ is analytically linearizable. This condition has been weakened to the so-called $(\omega)$-condition by Bruno \cite{Bruno} (we shall not recall here its definition since we shall not use it). Nevertheless, the control of these small divisors is not sufficient to ensure the analyticity of the normalizing transformation to a (non-linear) normal form. Indeed, there are well known situation for which, although the small divisors are bounded away from the origin, the normalizing transformation is anyway divergent at the origin \cite{Mart1}. Finally, A.D. Bruno \cite{Bruno} found a necessary and sufficient condition on the formal normal form -condition $(A)$- that ensures that if the semi-simple linear part satisfies the diophantine condition $(\omega)$, then any analytic perturbation of $S$ that has a formal normal form satisfying $(A)$ can be analytically normalized. This is a generalization of H. R\"ussmann's Hamiltonian version \cite{russmann-67}. These works have been generalized to several commuting vector fields in several directions \cite{Ito1,Stolo-ihes,stolo-annals, zung-birkhoff, zung-nf, Ito3} in the spirit of J. Vey \cite{vey-iso,vey-ham}.

In \cite{stolo-lombardi}, a broader notion of ``small divisors'' associated to a linear vector field $S$ is defined. They are defined as the square roots of the nonzero eigenvalues of the ``box operator'' $\square_{\cH_k}:=d_{0|\cH_k}d_{0|\cH_{k}}^*$, for all $k>2$. When $S$ is a nilpotent linear vector field in dimension two and three, G. Iooss and E. Lombardi \cite{lombardi-nf} computed the corresponding ``generalized resonances'' as well as the small divisors. Recently, P. Bonckaert and F. Verstringe succeeded in \cite{verstringe-nil}, 
to estimate the generalized small divisors for (almost) any nilpotent linear part. 
It is there shown that they are always bounded away from the origin. 

One of the remaining questions was whether a Takens normal form (in dimension 2) could be obtained systematically by a convergent transformation. It took almost forty years to be answered. In dimension $2$, after an attempt to solve this question by X. Gong \cite{gong-nil}, finally E. Str{\'o}{\.z}yna and H. {\.Z}o{\l}adek \cite{zoladek-nil} showed that any holomorphic nonlinear perturbation of $y\pp{x}$ has a germ of biholomorphism  that conjugate it to a Takens normal form~: $(y+f(x,y))\pp{x}+ g(x,y)\pp{y}$ is analytically conjugate to $(y+k(x))\pp{x}+ l(x)\pp{y}$; here $f,g,k,l$ are analytic germs of order $\geq 2$ at the origin. A more geometric proof was given in \cite{loray-zoladek}. 
It was shown also by E. Str{\'o}{\.z}yna and H. {\.Z}o{\l}adek that the equivalent theorem 
in a higher dimensional setting is false \cite{zoladek-nil-div,zoladek-nil-multidim}.
 
For the similar problem with the ``scalar product normal form'', we emphasize, that in dimension $2$, the ``small divisors'' tend to infinity with the degree of homogeneity. This phenomenon has a regularizing effect on the solution of the conjugacy equation to a normal form. This corresponds, somehow, to the situation of the Poincar\'e domain for non-zero semi-simple linear part, where no requirements are needed on the perturbation for it to have an analytic transformation to a normal form. 

In higher dimension an example in \cite{lombardi-nf} shows that the ``small divisors'' may not tend to infinity with the degree of homogeneity~: they may accumulate finite numbers as well as infinity. The higher dimensional situation for the nilpotent case corresponds somehow to the ``Siegel domain''~: one needs to impose conditions on the formal normal form in order to obtain the convergence of such a transformation.\\
$\;$
 
The aim of this article is to investigate, in dimension $n\geq 3$, the holomorphic conjugacy to a normal form problem for an higher order analytic perturbation of nilpotent linear vector field. We shall define a nilpotent version of Bruno's condition $(A)$. We shall state and prove a nilpotent version of the ``sufficiency part'' of Bruno's theorem~: in dimension $\geq 3$, if an analytic higher perturbation of a linear nilpotent $N$ has its formal normal form that satisfies this condition, then there exists an analytic transformation to its normal form.

\subsection{Main result}
In this section we formulate our main result using the notations introduced in the introduction. We first introduce some extra terminology.
Let $S$ be a linear vector field. We denote by $\widehat\cO_n^S$ the ring of formal first integrals of $S$, that is the set of formal power series $\hat f$ for which the Lie derivative $\cL_S(\hat f)=0$ vanishes.
A nilpotent matrix is said to be {\it regular} if its Jordan normal form does not contain zero blocks. In other words, a nilpotent endomorphism is regular if there is no invariant subspace upon which the operator acts as a zero operator. We shall denote by $R_{\geq 2}$ a germ of holomorphic vector field (or function, map depending on the context) vanishing at order $\geq 2$ at the origin.
\begin{theorem}\label{main-thm}
Let $N$ be a regular nilpotent linear vector field of $\Bbb C^n$ and assume that $\{N,N^*,[N,N^*]\}$ is a $\slt$-triple.
Let $V=N+R_{\geq 2}$ be a germ of holomorphic vector field at the origin of $\Bbb C^n$. Assume that its formal normal form (as defined by the iteration procedure described in the introduction) has the form 
\begin{equation}\label{good-nf}
V=N+fN^*\quad \text{where }\quad f\in\widehat \cO_n^N\cap\widehat \cO_n^{N^*}.
\end{equation}
Then, its associated formal diffeomorphism converges at the origin. 
\end{theorem}

This is the main result of our article and should be understood as a {\bf ``complete integrability'' theorem} along the same lines as (and in fact a continuation of) results by A.D. Bruno \cite{Bruno} (see also \cite{Mart1,zung-nf}), and by the first author \cite{Stolo-ihes, stolo-annals}. Indeed, in these articles, the authors consider a (germ of ) vector field (or commuting families of vector fields) $V=S+R_{\geq 2}$ that are holomorphic nonlinear perturbation of a semi-simple linear vector field $S=\sum_{i=1}^n\la_ix_i\pp{x_i}$. The main (easiest) theorem of Bruno can be stated as follows~: Assume that the normal form of $V$ is of the form $S+fS+g\bar S$ where $\bar S:=\sum_{i=1}^n\bar\la_ix_i\pp{x_i}$ and $f,g\in \widehat{\cO}_n^S$, $f(0)=g(0)=0$ (this is Bruno's ``condition $(A)$''). If $S$ satisfies Bruno's small divisors condition $(\om)$, then its associated formal biholomorphism is convergent at the origin and it conjugates $V$ to a (convergent) normal form.

Theorem \ref{main-thm} can be seen as a {\bf nilpotent version of Bruno's theorem}. Indeed, $\bar S$ plays the role of $N^*$ and we have $\widehat{\cO}_n^S= \widehat{\cO}_n^S\cap \widehat{\cO}_n^{\bar S}$. In our normal form $(\ref{good-nf})$, there is no equivalent term to $fS$ since the natural one would be $fN$ which is not a normal form w.r.t. to $N$, i.e. $fN\notin \Ker ad_{N^*}$.
\begin{remark}
One of the key points of \cite{verstringe-nil} is that, if the linear nilpotent vector field $S$ is regular, then there exists a linear change of coordinates $L$ of $\Bbb C^n$ such that $N:=L_*S$, $N^*$, the vector field defined as the adjoint of the differential operator acting on polynomials of $N$ w.r.t. the scalar product used above, and $H=[N, N^*]$ define an $\slt$-triple. Moreover, $\{\text{ad}_{N},\text{ad}_{N^*}, \text{ad}_{N}\text{ad}_{N^*}-\text{ad}_{N^*}\text{ad}_{N}\}$ is also an $\slt$-triple.
In the sequel, we shall assume that the linear change of coordinates $L$ has been applied and 
hence, that both $\{N,N^*,[N,N^*]\}$ and $\{\text{ad}_{N},\text{ad}_{N^*},[\text{ad}_{N},\text{ad}_{N^*}]\}$ are $\slt$-triples. 
\end{remark}
\begin{remark}\label{rem-sl2-cv}
If $V=N+R_{\geq 2}$ is a germ of holomorphic vector field at the origin of $\Bbb C^n$ that can be embedded into an analytic $\slt$-triple, that is if there are germs $\tilde V=N^*+h.o.t.$, $V'=H+h.o.t.$ of holomorphic vector fields such that $\{V,\tilde V, V'\}$ is (Lie-)isomorphic to $\slt$, then  $V$,$\tilde V$ and $V'$ are simultaneously and holomorphically linearizable. Indeed, at the formal level, this follows from \cite{Hermann-semi-simple} while the analytic result follows from \cite{stern-semi-simple,Kushnirenko}. It should be emphasized that this result does not hold in the smooth category \cite{stern-semi-simple,Ghys}. 
\end{remark}
\begin{remark}\label{rem-sl2-form}
If $V=N+R_{\geq 2}$, a germ of holomorphic vector field at the origin of $\Bbb C^n$, can be embedded into a formal $\slt$-triple, that is if there are formal vector fields $\tilde V=N^*+h.o.t.$, $V'=H+h.o.t.$ such that $\{V,\tilde V, V'\}$ is (Lie-)isomorphic to $\slt$, then  $V$,$\tilde V$ and $V'$ are simultaneously formally linearizable. According to \cite{verstringe-nil}[p.2223], the ``small divisors'' are bounded away from the origin. Then, according to \cite{stolo-lombardi}[theorem 5.8], $V$ is analytically linearizable.
\end{remark}

There are other results concerning the problem of classification of perturbations not of nilpotent vector fields but rather of quasihomogenous vector fields with a nilpotent linear part at the origin. They are all in dimension $2$ (see for instance \cite{Moussu-nil, schafke-nil}) and are not immediately concerned with holomorphic conjugacy (in a neighborhood) at the origin to a normal form.

\subsection{Geometric interpretation}

According to Weitzenb\"ock's theorem \cite{weitzenbock}, the ring of formal first integrals $\widehat\cO_n^N$ of any nilpotent linear vector field $N$ is finitely generated over $\Bbb C$ (see \cite{nowicki}[theorem 6.2.1] for this formulation). 

The ring of common formal first integrals $\widehat\cO_n^N\cap \widehat\cO_n^{N^*}$ is hence finitely generated over $\Bbb C$. Let $P_1,\ldots,P_r$ be a set of generating polynomials. If they are algebraically dependent, then there exists polynomials $Q_1,\ldots, Q_l$ on $\Bbb C^r$ such that $Q_i(P_1,\ldots, P_r)=0$ for all $i$. Let us consider 
$$
{\mathcal C}=\{z\in \Bbb C^r\;|\;Q_i(z)=0,\;i=1,\ldots, l\}
$$
the zero locus of these polynomials. If the $P_j$'s are algebraically independent, then ${\mathcal C}=\Bbb C^r$.
Consider the map
\begin{eqnarray*}
\pi~: \Bbb C^n &\rightarrow & ({\mathcal C}, 0),\\
x &\mapsto &  (P_1(x),\ldots,P_r(x)).
\end{eqnarray*}

We apply our main result and assume that the vector field $V$ has been conjugated into a normal form by a germ of holomorphic transformation at the origin. In these holomorphic coordinates, we have
$$
V= N+fN^*
$$
where $f\in \cO_n^N\cap \cO_n^{N^*}$ is a germ of holomorphic function. Let $b\in ({\mathcal C},0)$ be a point in the image of $\pi$. Then, 
\begin{enumerate}
\item {\bf $V$ is tangent to the level set of $\Sigma_b:=\pi^{-1}(b)$}, intersected with a neighborhood of the origin. Indeed, both $N$ and $N^*$ are tangent to this level set and the function $f$ is constant on this level set. 
\item the restriction of the vector field $V$, $V_{|\Sigma_b}$, {\bf is the restriction to $\Sigma_b$ of a linear vector field}. 
\item the restriction $V_{|\Sigma_b}$, {\bf belongs to the restriction an $\slt$-action}\footnote{We thank Nguyen Tien Zung for this remark}, namely $\{N_{|\Sigma_b},N^*_{|\Sigma_b}, H_{|\Sigma_b}\}$ (compare with remarks \ref{rem-sl2-cv} and \ref{rem-sl2-form}).
\item Moreover {\bf this holds for every fiber }(in a neighborhood of the origin in $\Bbb C^n$) within a neighborhood of the origin in ${\mathcal C}\cap \Im\pi$. 
\end{enumerate}
As a by-product, we also obtain that the zero fiber of $\pi$ is invariant, the restriction to which $V$ equals to the restriction of the linear nilpotent field $N$. As illustrated in the next example, this last result could have obtained under a much weaker assumption using \cite{stolo-lombardi}[theorem 5.6] combined with the ``no-small-divisors'' statement of \cite{verstringe-nil}.

\begin{exam}
Consider the three dimensional linear vector field~:
$$
N:= y\frac{\partial}{\partial x}+z\frac{\partial}{\partial y}.
$$
We have 
$$
N^*:= x\frac{\partial}{\partial y}+y\frac{\partial}{\partial z}.
$$
Let us set
$$
h=xz-\frac{y^2}{2}.
$$
It follows that, $N(h)=0=N^*(h)$ and $\widehat \cO_3^N\cap\widehat \cO_3^{N^*}=\Bbb C[[h]]$.
A formal normal form of a nonlinear perturbation of $N$ is a formal vector field of the form
\begin{eqnarray}\label{nfO3}
\dot x & = & y + xP_1(x,h)\\
\dot y & = & z + yP_1(x,h)+xP_2(x,h)\nonumber\\
\dot z & = & zP_1(x,h)+yP_2(x,h)+P_3(x,h)\nonumber
\end{eqnarray}
where the $P_i$'s are formal power series \cite{lombardi-nf}.

We apply our main result~: Let $V=N+R_{\geq 2}$ be a germ of holomorphic vector field that is a nonlinear perturbation of $N$. Assume that the normal form of $V$ is of the form~: 
$$
\hat\Phi_*V=N+\hat f\left(xz-\frac{y^2}{2}\right)N^*
$$
where $\hat f$ is a formal power series of one variable (that is $P_1=P_3=0$, $P_2(x,h)=\hat f(h)$). Then, $\hat\Phi$ defines a germ of holomorphic change of variables $\Phi$ at the origin of $\Bbb C^3$ and $\hat f$ defines a germ of holomorphic function $f$ at $0\in \Bbb C$ such that 
$$
\Phi_*V=N+ f\left(xz-\frac{y^2}{2}\right)N^*.
$$
In these holomorphic coordinates, the vector field is tangent to each fiber 
$$
\Sigma_b=\left\{(x,y,z)\in (\Bbb C^3,0)\;|\;xz-\frac{y^2}{2}=b\right\}
$$ 
in a neighborhood of the origin, $b$ being sufficiently small. In particular, in these new holomorphic coordinates, the vector field is tangent to $\Sigma_0$ and its restriction to it is equal to (the restriction of) $N$.
We emphasize that the last point can be achieved under a much weaker assumption~: Assume that the formal normal form is such that $P_i(x,0)=0$, $i=1,2, 3$. Then, according to \cite{stolo-lombardi}[theorem 5.6] applied to ${\mathcal I}=(h)$, there are holomorphic coordinates for which $\Sigma_0$ in an invariant set of the initial vector field $X$. Furthermore, its restriction to $\Sigma_0$ is equal to $N$. 
\end{exam}

\subsection{Sketch of the proof}

The proof of the main theorem is done using a Newton scheme. We assume that $V=NF_m+R_{\geq m+1}$ is normalized up to order $m=2^k$. By assumption, the partial normal form $NF_m$ is of the form $N+f_mN^*$ where $f_m\in \widehat \cO_n^N\cap\widehat \cO_n^{N^*}$ is a polynomial of degree $\leq m-1$. There is a unique polynomial vector field $U$ of the form $U=d_0^*(V)$, of order $\geq m+1$ and of degree $\leq 2m$ such that the diffeomorphism $(\id+U)^{-1}$ normalizes $V$ up to order $2m$. The vector field $U$ solves an equation of the form~:
$$
J^{2m}([NF_m,U])=B
$$
for some known $B$, in the range of the box operator $\square$. Here, $J^{2m}$ denotes the truncation at degree $2m$ of the Taylor expansion at the origin.
We then need to estimate $U$ in terms of $B$. For that purpose, we shall first define a suitable analytic norm (depending on some ``radius'' $r$, see section \ref{section-norms}) based on the scalar product we mentioned above (see section \ref{sp-section}). Following (\ref{al:mainoperatorequation}), we shall then re-write the previous equation as 
$$
V=\left( \id +Q_1+Q_2 \right)^{-1}\boxtildeinv B,
$$
where $\boxtilde$ denotes the restriction of the operator $\square:=d_0d_0^*$ to its image and 
$Q_1,Q_2$ are operators defined by (\ref{def-Q}) and (\ref{def-P}), depending on $f_m$. Using {\bf all properties of $\slt$-triples} (see section \ref{sec:machinery}), we shall be able to estimate the normal of $Q_1,Q_2$ and $\boxtildeinv$. As a consequence, we shall obtain an estimate of $V$ in terms of $B$, and then an estimate of $U$ in terms of $B$ (see proposition \ref{prop-estim-cohom}). Then the conjugacy of $V$ by $(\id+U)^{-1}$ is normalized up to order $2m$~: $(\id+U)^{-1}_*V=NF_{2m}+R_{\geq 2m+1}$ where $NF_{2m}:=N+f_{2m}N^*$ is a normal form. Choosing $R<r$ in an appropriate way, we shall show that the new normal form $NF_{2m}$ as well as the new remainder $R_{\geq 2m+1}$ satisfy the same estimate w.r.t. the $R$-norm that $NF_m$ and $R_{\geq m+1}$ satisfy w.r.t. the $r$-norm (proposition \ref{induction}). This allow to do an induction process. The key point (lemma \ref{small-div}) is that the sequence of successive radii $\{R_k\}$ converges to a positive number. This is due to the sharp estimate of the solution $U$ we obtain (for each step). A classical argument then shows that the sequence of (partial) normalizing diffeomorphism converges to a genuine holomorphic diffeomorphism that conjugates $V$ to a normal form.
\subsection{Convergence to the normal form requires a condition}

We don't know whether condition $(\ref{good-nf})$ is necessary for the convergence of the normalizing transformation to hold (as it is the case of Bruno's condition (A) in the semi-simple case). However, Str{\'o}{\.z}yna and {\.Z}o{\l}adek  have shown in \cite{zoladek-nil-div}, that the normalizing transformation of the vector field 
$$
(*)\left\{\begin{array}{ccl} \dot x & = & y-x^2+x^3\\ \dot y &=& z+2x^4\\ \dot z &=&0\end{array}\right.
$$
to its (generalized) Takens normal form is divergent at the origin. Although, we shall not investigate the dictionary between different ``styles'' of normal form, let us show that the ``scalar product'' normal form of $(*)$ doesn't satisfy $(\ref{good-nf})$. Indeed, we have the following
\begin{lemma}
If $f(x,y,z)$ is a formal first integral of $(*)$, then $f$ is a function of $z$, only.
\end{lemma}
\begin{proof}
Let $\cL$ be the vector field associated to $(*)$ and let $N:=y\frac{\partial}{\partial x}+z\frac{\partial}{\partial y}$ be its linear part at $0$. Let us assume, by induction, that $f=\sum_{i=1}^kf_iz^i+\sum_{j\geq k+1}g_{j}$  where $g_j$ is homogeneous of degree $j$ in {\bf all variables}, and $f_i$'s are complex coefficients. Hence, 
we have $0=\cL(f)=\cL(g_{k+1}+g_{k+2}+\text{h.o.t})$ so $N(g_{k+1})=0$. Therefore, $g_{k+1}$ is a polynomial $P_{k+1}$ of $z$ and $h=xz-\frac{y^2}{2}$. Let us show that $P_{k+1}=\alpha_{k+1}z^{k+1}$ and that $g_{k+2}=\alpha_{k+2}z^{k+2}+ \beta_{k+2}z^{k}h$. Indeed, the Taylor polynomial of degrre $k+2$ of $\cL(f)$ is
$$
N(g_{k+2})-x^2\frac{\partial g_{k+1}}{\partial x}=0= N(g_{k+2})-x^2z\frac{\partial P_{k+1}}{\partial h}.
$$
Setting $z=0$, we obtain $y\frac{\partial g_{k+2}}{\partial x}(x,y,0)=0$. Therefore, $g_{k+2}=z\tilde g_{k+2}$ and $N(\tilde g_{k+2})-x^2\frac{\partial P_{k+1}}{\partial h}=0$.
 Let us write $x^2\frac{\partial P_{k+1}}{\partial h}=\alpha x^2h^m+x^2zp(z,h)$ for some polynomial $p$ and a constant $\alpha$. Since $x^2h^m\in Ker N^*$, $x^2\frac{\partial P_{k+1}}{\partial h}$  can belong to the range of $N$ only when $\alpha =0$. Thus, $x^2\frac{\partial P_{k+1}}{\partial h}=x^2zp(z,h)$, $\tilde g_{k+2}=zG$ and  $N(G)=x^2p(z,h)$. We continue this process. We finally obtain that $g_{k+2}=z^{k}G_2$ and $N(G_2)=\alpha x^2$ for some quadratic polynomial $G_2$. 
Since $N^*(x^2)=0$, we have $\alpha=0$ and thus $G_2$ is a first integral~: $G_2=\alpha z^2+\beta h$. This means that $g_{k+2}= \alpha_{k+2}z^{k+2}+ \beta_{k+1}z^{k}h$. As a consequence, $N(g_{k+2})=0$, so that $\frac{\partial P_{k+1}}{\partial h}=0$ and $P_{k+1}=
\alpha_{k+1}z^{k+1}$.
\end{proof}
As a consequence, system $(*)$ can not be formally conjugate to a system satisfying $(\ref{good-nf})$.
Indeed, if it was, then its formal first integral would be a push-forward of a first integral of the second system. 
Since $(\ref{good-nf})$ has $h=xz-\frac{y^2}{2}$ has first integral, then $(*)$ would have a formal first integral of the form $h\circ\Phi$, for some formal diffeomorphism $\Phi$, tangent to identity. But such a function cannot be a function of the variable $z$ only.\\
$\;$\\
\noindent
{\bf Acknowledgement }
{\it Part of the is work was done while the second author spent 6 months on a postdoctoral position provided by University of Nice Sophia Antipolis during academic year 2011-2012. The first author thank M. Gazor, J. Murdock and J. Sanders for interesting email exchanges on  uniqueness of formal normal/hypernormal forms}. 

\section{Background}

For any non-negative integer $k$, $\mathcal{P}_k$ (resp. $\mathcal{V}_k$) denotes the space of homogeneous polynomials (resp. polynomial vector fields) of degree $k$. We shall also write $\mathcal{P}_{k,m}:=\bigoplus_{j=k}^m\mathcal{P}_j$  (resp. $\mathcal{P}_{\geq k}:=\bigoplus_{j=k}^{\infty}\mathcal{P}_j$) as well as  $\mathcal{V}_{k,m}:=\bigoplus_{j=k}^m\mathcal{V}_j$ (resp. $\mathcal{V}_{\geq k}:=\bigoplus_{j=k}^{\infty}\mathcal{V}_j$)

\subsection{Scalar product}\label{sp-section}

We introduce first some notation. 
Let $\alpha\in \N^n$, we shall write $\alpha!:=\alpha_1!\ldots \alpha_n!$. We define $\mathcal{P}_{\delta}$ to be the space of homogeneous polynomials of degree $\delta$ in $n$ variables $x_1,\ldots,x_n$. We first consider the scalar product on $\cP_{\de}$ introduced by E. Fischer \cite{fischer} and G. Belitskii \cite{Belitskii-formal}~:
\begin{lemma}\label{belitskii}
The scalar product defined on $\cP_{\de}$~:
$$
 \Big<\sum_{|\alpha|=\delta}a_{\alpha}x^{\alpha},\sum_{|\beta|=\delta}c_{\beta}x^{\beta}\Big>_{B,\delta
} : =\sum_{|\alpha|=\delta}a_{\alpha}\bar{c}_{\alpha}\alpha!
$$
has the following property~: the operator $\frac{\partial}{\partial x_j}$ is adjoint to the multiplication by $x_j$ operator, that is  $<\frac{\partial f}{\partial x_j}, g>_{B,\de}= <f, x_jg>_{B,\de +1}$ and  $<g,\frac{\partial f}{\partial x_j}>_{B,\de}= <x_jg, f>_{B,\de+1}$ for all $(f,g)\in\cP_{\de+1}\times \cP_{\de}$.
\end{lemma}

In \cite{stolo-lombardi}, a variant of this scalar product was introduced, namely~: 
\begin{align}
\label{eq:inp}
 \Big<\sum_{|\alpha|=\delta}a_{\alpha}x^{\alpha},\sum_{|\beta|=\delta}c_{\beta}x^{\beta}\Big>_\delta
 =\sum_{|\alpha|=\delta}a_{\alpha}\bar{c}_{\alpha}\frac{\alpha!}{|\alpha|!}.
\end{align}
Let $\|.\|_{\de}$ be its associated norm. Since homogeneous polynomial of different degree are orthogonal to each other, we are able to define the norm of formal power series $f=\sum_{k\geq 0} f_k$ where $f_k$ is a homogeneous polynomial of degree $k$ as
$$
\|f\|^2:=\sum_{k\geq 0}\|f_k\|^2_k.
$$
It has the some important properties that are summarized in the following proposition~:
\begin{proposition}\cite{stolo-lombardi}[proposition 3.6-3.7]
Let $f,g$ be formal power series on $\Bbb C^n$. Let us write $f=\sum_{k\geq 0} f_k$ where $f_k$ is a homogeneous polynomial of degree $k$. Then,
\begin{itemize}
\item $\|fg\|\leq \|f\|\|g\|$
\item $f$ defines a germ of holomorphic function at the origin if and only if there exist positive constants $M,C$ such that $\|f_k\|_k\leq MC^k$.
\end{itemize}
\end{proposition}
This scalar product induces a scalar product on $\mathcal{V}_\delta$, the space of homogeneous vector fields of degree $\delta$ in $n$ variables. Such a vector field $V$ can be written as $V=\sum_{k=1}^{n} V_k
\pp{x_k}$, where $V_k\in\mathcal{P}_\delta$.  
\begin{align}
\label{eq:inpvv}
\Big<\sum_{k=1}^{n} V_k \pp{x_k},\sum_{k=1}^{n} W_k \pp{x_k}\Big>_\delta= \sum_{k=1}^{n} \left<V_k,W_k\right>_{\delta}.
\end{align}

\subsection{The analytic norm}\label{section-norms}
Let us define, for $\delta\in \Bbb N$, $c_\delta:=\big(\sum_{|Q|=\delta} \frac{|Q|!}{Q!}\big)^{1/2}$.  Let $f$ be a formal power series vanishing at the origin. It can be written as a sum of homogeneous polynomials $f=\sum_{\delta\geq 1}f_\delta$, where $f_\delta\in \mathcal{P}_\delta$. We define, for $r\geq 0$,
\begin{equation}\label{def-normr}
 |f|_r:=\sum_{\delta\geq 1}||f_\delta||_\delta c_\delta r^\delta,
\end{equation}
in this formula $||f_\delta||_\delta$ is the norm associated to the scalar product (\ref{eq:inp})~:
$$
||f_\delta||_\delta^2=||\sum_{|k|=\delta}f_k x^k ||_\delta^2=\sum_{|k|=\delta}|f_k|^2 \frac{k!}{|k|!}.
$$
In a similar manner, we extend this norm to the space of vector fields using $(\ref{eq:inpvv})$.
\begin{lemma}
\label{lem:cdelta_afschatting}
 $c_{\delta}=n^{\delta/2}$.
\end{lemma}
\begin{proof}
 This follows immediately by choosing $z_i=1$, $1\leq i\leq n$ in the formula (see, for instance, \cite{lombardi-nf}[lemma A.1.])
 \begin{align*}
 (z_1+\ldots+z_n)^\delta=\sum_{|Q|=\delta}\frac{|Q|!}{Q!}z^Q.
\end{align*}
\end{proof}
\begin{lemma}\label{prod-norm-r}
Let $f,g$ be formal power series. Then $|fg|_r\leq |f|_r|g|_r$.
\end{lemma}
\begin{proof}
We have 
$$
|fg|_r= \sum_{k\geq 0} \left\|\{fg\}_k\right\|_k c_k r^k = \sum_{k\geq 0} \left\|\sum_{i+j=k}f_ig_j\right\|_k c_k r^k.
$$
We recall that $\{fg\}_k$ denotes the homogeneous part of degree $k$ of $fg$ in the Taylor expansion at the origin.
Since the scalar product $(\ref{eq:inp})$ defines a Banach algebra norm, we have $\left\|\sum_{i+j=k}f_ig_j\right\|_k\leq  \sum_{i+j=k}\|f_i\|_i\|g_j\|_j$. Since $c_k=n^{k/2}=c_ic_j$ if $i+j=k$, then
$$
|fg|_r\leq \sum_{k\geq 0}\sum_{i+j=k}\|f_i\|_ic_ir^i\|g_j\|_jc_jr^j=\left(\sum_{i\geq 0} \left\|f_i\right\|_i c_i r^i\right)\left(\sum_{j\geq 0} \left\|g_j\right\|_j c_j r^j\right)
$$
\end{proof}
\begin{lemma}
 Suppose that $|u|_R\leq r$, then $|f\circ u|_R \leq |f|_r$
\end{lemma}
\begin{proof}
Let $f=\sum_{Q\in \Bbb N^n}f_Qx^Q$ be the Taylor expansion of $f$ at the origin.
 \begin{eqnarray*}
  |f\circ u|_R&=&\left|\sum_{\delta\geq 1}\sum_{|Q|=\delta} f_Q.(u(x))^Q\right|_R \leq \sum_{\delta\geq 1}\sum_{|Q|=\delta} |f_Q|.|(u(x))^Q|_R\\
&\leq & \sum_{\delta\geq 1}\sum_{|Q|=\delta} |f_Q|.|(u(x))|_R^{|Q|}\leq \sum_{\delta\geq 1}\sum_{|Q|=\delta} |f_Q| r^{|Q|}\\
&\leq &\sum_{\delta\geq 1}\sum_{|Q|=\delta} |f_Q|\sqrt{\frac{Q!}{|Q|!}}\sqrt{\frac{|Q|!}{Q!}} r^{|Q|}\\
& &\hphantom{\leq}\text{(Apply Cauchy Schwartz)}\\
&\leq& \sum_{\delta\geq 1}\Big(\sum_{|Q|=\delta} |f_Q|^2\frac{Q!}{|Q|!}\Big)^{1/2}\Big(\sum_{|Q|=\delta} \frac{|Q|!}{Q!} r^{2|Q|}\Big)^{1/2}\\
&\leq &\sum_{\delta\geq 1}\Big(\sum_{|Q|=\delta} |f_Q|^2\frac{Q!}{|Q|!}\Big)^{1/2}c_\delta r^\delta =\sum_{\delta\geq 1} ||f_\delta||_\delta c_\delta r^\delta\\
&=&|f|_r.
\end{eqnarray*}
\end{proof}

\begin{lemma}\label{cauchy}
 Let $f$ be a polynomial of degree at most $m$, then $|\p{f}{x_i}|_r\leq \frac{m}{r}|f|_r$.
\end{lemma}
\begin{proof}
We remark first that it readily follows from Lemma~\ref{lem:cdelta_afschatting} that $c_{\delta-1}\leq c_\delta$.
We estimate
\begin{eqnarray*}
 \Big|\p{f_\delta}{x_j}\Big|_{\delta-1}^2&\leq & \sum_{|Q|=\delta}|f_Q|^2 q_j^2\frac{Q!}{q_j(\delta-1)!}\leq \sum_{|Q|=\delta}|f_Q|^2\frac{Q!}{\delta!}\frac{\delta!q_j}{(\delta-1)!}\\
&\leq &\sum_{|Q|=\delta}|f_Q|^2\frac{Q!}{\delta!}(\delta q_j)\leq \delta^2\sum_{|Q|=\delta}|f_Q|^2\frac{Q!}{\delta!}. 
\end{eqnarray*}
Hence we have $|\p{f_\delta}{x_i}|_{\delta-1} \leq \delta|f_\delta|_\delta$. As a consequence, we obtain the following estimates~:
\begin{eqnarray*}
 \Big|\p{f}{x_i}\Big|_r&= &\sum_{\delta=1}^{m}\left|\p{f_\delta}{x_i}\right|_{\delta-1}c_{\delta-1}r^{\delta-1} \leq \sum_{\delta=1}^{m}\delta|f_\delta|_{\delta}c_{\delta-1}r^{\delta-1}\\
&\leq & \frac{m}{r}\sum_{\delta=1}^{m}|f_\delta|_{\delta}c_{\delta-1}r^{\delta} \leq \frac{m}{r}\sum_{\delta=1}^{m}|f_\delta|_{\delta}c_{\delta}r^{\delta}\\
&\leq & \frac{m}{r}|f|_r.
\end{eqnarray*}
\end{proof}

\subsection{Normal forms}\label{nf-section}

Let $F=\sum_{i=1}^n\left(\sum_{j=1}^nf_{i,j}x_j\right)\frac{\partial}{\partial x_i}$ be a linear vector field on $\Bbb C^n$. It acts by derivation as a linear map from $\cP_{\de}$ to itself. According to lemma \ref{belitskii}, the adjoint of this operator w.r.t. Belitskii scalar product is the derivation 
$$
F^*:= \sum_{j=1}^n\left(\sum_{i=1}^n\bar f_{i,j}x_i\right)\frac{\partial}{\partial x_j}.
$$
The adjoint w.r.t. the scalar product $(\ref{eq:inp})$ is the same. Let $\widetilde{N}$ be a nilpotent linear vector field in $\Bbb C^n$.
 Let us first recall the fundamental result that links the $\slt$-triple actions and adjoint operators with respect to the scalar product. 
\begin{prop}\cite{verstringe-nil}
Assume that $\ti N$ is regular. Then, there exists a linear change of coordinates $L$ such that $N=L_*\widetilde{N}$ and its adjoint $N^*$ generate an $\slt$-triple. More precisely, the differential operators $X=N^*$, $Y=N$ and $H=N^{*}N-NN^{*}$, acting on germs of holomorphic functions or formal power series, fulfill the $\slt$-relations from formula (\ref{al:sl2relations}).  Furthermore, $N$ acts
on vector fields of $V\in\mathcal{V}_\delta$, $\de\geq 2$, as 
$$
d_{0,\delta}(V):=[N,V],
$$ 
where $[.,.]$ denotes the Lie bracket of vector fields. Then, its adjoint $d_{0,\delta}^*$ w.r.t. to the scalar product $(\ref{eq:inpvv})$ is $d_{0,\delta}^*(V):=\left[N^*,V\right]$
and  the triple $X=d_{0,\delta}^*$,
$Y=d_{0,\delta}$, $H=[d_{0,\delta}^*,d_{0,\delta}]$ satisfies the $\slt$-relations from formula (\ref{al:sl2relations}).
\end{prop}
The operator $d_0$ is called the {\bf cohomological operator} and we shall define the {\bf box operator} of degree $\delta$, to be \begin{equation}\label{bef-box}\square_\delta:=d_{0,\delta}d_{0,\delta}^*.\end{equation}
We recall the main formal normal form result~: 
\begin{prop}\cite{stolo-lombardi}
Let $X:=N+R_{\geq 2}$ be a nonlinear holomorphic (or formal) perturbation of the nilpotent linear vector field $N$ in a neighborhood of the origin of $\Bbb
C^n$. Then, there exists a formal
diffeomorphism $\hat\Phi$ tangent to the identity that
 conjugates $X$ to a formal normal form; that is $\hat \Phi_*X-N\in \text{ Ker ad}_{N^*}$
Moreover, there exists a unique normalizing diffeomorphism $\Phi=\Id+\Ucal $
such that $\Ucal$ has a zero projection on the kernel of $d_0=[S,.]$.
\end{prop}
\begin{defi}
The normal form of $X$ obtained by conjugacy of the unique normalizing diffeomorphism $\Phi=\Id+\Ucal $
such that $\Ucal$ has a zero projection on the kernel of $d_0=[S,.]$, will be called {\bf the} normal form of $X$.
\end{defi}


%% file: freek-laurent-final.bbl
\begin{thebibliography}{ETB{\etalchar{+}}87}

\bibitem[Arn80]{Arn2}
V.I. Arnold.
\newblock {\em {Chapitres suppl\'{e}mentaires de la th\'{e}orie des
  \'{e}quations diff\'{e}rentielles ordinaires}}.
\newblock Mir, 1980.

\bibitem[Bel79]{Belitskii-formal}
G.~R. Belitskii.
\newblock Invariant normal forms of formal series.
\newblock {\em Funct. Anal. Appl.}, 13(1):46--47, 1979.

\bibitem[Bel82]{Belitskii-filtering}
G.~R. Belitskii.
\newblock {\em Normal forms relative to a filtering action of a group}, volume
  Transactions of the {M}oscow {M}athematical {S}ociety, 1981, {I}ssue 2, pages
  1--39.
\newblock American Mathematical Society, 1982.
\newblock A translation of Trudy Moskov. Mat. Ob{\v{s}}{\v{c}}. {{\bf{4}}0}
  (1979), Trans. Moscow Math. Soc. {{\bf{1}}981}, no. 2 (1982).

\bibitem[Bou90]{Lie78-bourbaki}
N.~Bourbaki.
\newblock {\em {Groupes et Alg\`{e}bre de Lie, Chapitres 7 et 8}}.
\newblock Masson, Paris, 1990.

\bibitem[Bru72]{Bruno}
A.D. Bruno.
\newblock {Analytical form of differential equations}.
\newblock {\em Trans. Mosc. Math. Soc}, 25,131-288(1971); 26,199-239(1972),
  1971-1972.

\bibitem[BV12]{verstringe-nil}
P.~Bonckaert and F.~Verstringe.
\newblock Normal forms with exponentially small remainder and gevrey
  normalization for vector fields with a nilpotent linear part.
\newblock {\em Ann. Inst. Fourier (Grenoble)}, 62(6):2211--2225, 2012.

\bibitem[CDS04]{schafke-nil}
M.~Canalis-Durand and R.~Sch{\"a}fke.
\newblock Divergence and summability of normal forms of systems of differential
  equations with nilpotent linear part.
\newblock {\em Ann. Fac. Sci. Toulouse Math. (6)}, 13(4):493--513, 2004.

\bibitem[CG97]{Ghys}
G.~Cairns and E.~Ghys.
\newblock The local linearization problem for smooth $sl(n)$-actions.
\newblock {\em Enseignement Math.}, 43, 1997.

\bibitem[CS86]{cushman-sanders}
R.~Cushman and J.~A. Sanders.
\newblock Nilpotent normal forms and representation theory of {${\rm sl}(2,{\bf
  R})$}.
\newblock In {\em Multiparameter bifurcation theory (Arcata, Calif., 1985)},
  volume~56 of {\em Contemp. Math.}, pages 31--51. Amer. Math. Soc., 1986.


\bibitem[Fis17]{fischer}
E.~Fischer.
\newblock \"{U}ber die {D}ifferentiationsprozesse der {A}lgebra.
\newblock {\em J. f\"ur Math. 148, 1-78.}, 148:1--78, 1917.

\bibitem[Gon95]{gong-nil}
X.~Gong.
\newblock Integrable analytic vector fields with a nilpotent linear part.
\newblock {\em Ann. Inst. Fourier (Grenoble)}, 45(5):1449--1470, 1995.

\bibitem[GR71]{Grauert-L1}
H.~Grauert and R.~Remmert.
\newblock {\em Analytische Stellenalgebren}.
\newblock Springer-Verlag, 1971.

\bibitem[GS68]{stern-semi-simple}
V.V. Guillemin and S.~Sternberg.
\newblock Remarks on a paper of {H}ermann.
\newblock {\em Trans. Amer. Math. Soc.}, 130:110--116, 1968.

\bibitem[GY12]{gazor-spectral}
Majid Gazor and Pei Yu.
\newblock Spectral sequences and parametric normal forms.
\newblock {\em J. Differential Equations}, 252(2):1003--1031, 2012.

\bibitem[Her68]{Hermann-semi-simple}
R.~Hermann.
\newblock The formal linearization of a semi-simple lie algebra of vector
  fields about a singular point.
\newblock {\em Trans. Amer. Math. Soc.}, 130:105--109, 1968.

\bibitem[IL05]{lombardi-nf}
G.~Iooss and E.~Lombardi.
\newblock Polynomial normal forms with exponentially small remainder for
  analytic vector fields.
\newblock {\em J. Differential Equations}, 212(1):1--61, 2005.

\bibitem[Ito89]{Ito1}
H.~Ito.
\newblock Convergence of {B}irkhoff normal forms for integrable systems.
\newblock {\em Comment. Math. Helv.}, 64:412--461, 1989.

\bibitem[Ito09]{Ito3}
H.~Ito.
\newblock Birkhoff normalization and superintegrability of {H}amiltonian
  systems.
\newblock {\em Ergodic Theory Dynam. Systems}, 29(6):1853--1880, 2009.

\bibitem[KOW96]{kokubu}
H.~Kokubu, H.~Oka, and D.~Wang.
\newblock Linear grading function and further reduction of normal forms.
\newblock {\em J. Differential Equations}, 132(2):293--318, 1996.

\bibitem[Kus67]{Kushnirenko}
A.G. Kushnirenko.
\newblock Linear-equivalent action of a semi-simple lie group in the
  neighbourhood of a stationary point.
\newblock {\em Funct. Anal. Appl.}, 1:89--90, 1967.

\bibitem[Lor06]{loray-zoladek}
F.~Loray.
\newblock A preparation theorem for codimension-one foliations.
\newblock {\em Ann. of Math. (2)}, 163(2):709--722, 2006.

\bibitem[LS10]{stolo-lombardi}
E.~Lombardi and L.~Stolovitch.
\newblock Normal forms of analytic perturbations of quasihomogeneous vector
  fields: Rigidity, invariant analytic sets and exponentially small
  approximation.
\newblock {\em Ann. Scient. Ec. Norm. Sup.}, pages 659--718, 2010.

\bibitem[Mar80]{Mart1}
J.~Martinet.
\newblock {Normalisation des champs de vecteurs holomorphes}.
\newblock {\em S\'{e}minaire Bourbaki}, 33(1980-1981),564, 1980.

\bibitem[MC88]{Moussu-nil}
R.~Moussu and D.~Cerveau.
\newblock {Groupes d'automorphismes de $(\Bbb C,0)$ et \'{e}quations
  diff\'erentielles $ydy+\cdots=0$}.
\newblock {\em Bull. Soc. math. France}, 16(1988),459-488, 1988.

\bibitem[Mos73]{moser-princeton}
J.~Moser.
\newblock {\em Stable and random motions in dynamical systems}.
\newblock Princeton University Press, Princeton, N. J.; University of Tokyo
  Press, Tokyo, 1973.
\newblock With special emphasis on celestial mechanics, Hermann Weyl Lectures,
  the Institute for Advanced Study, Princeton, N. J, Annals of Mathematics
  Studies, No. 77.

\bibitem[Mur03]{murdock-book}
J.~Murdock.
\newblock {\em Normal forms and unfoldings for local dynamical systems}.
\newblock Springer Monographs in Mathematics. Springer-Verlag, New York, 2003.

\bibitem[Mur04]{murdock-hyper}
James Murdock.
\newblock Hypernormal form theory: foundations and algorithms.
\newblock {\em J. Differential Equations}, 205(2):424--465, 2004.

\bibitem[Now94]{nowicki}
A.~Nowicki.
\newblock {\em Polynomial derivations and their rings of constants}.
\newblock Uniwersytet Miko\l aja Kopernika, Toru\'n, 1994.

\bibitem[R{\"u}s67]{russmann-67}
H.~R{\"u}ssmann.
\newblock \"{U}ber die {N}ormalform analytischer {H}amiltonscher
  {D}ifferentialgleichungen in der {N}\"ahe einer {G}leichgewichtsl\"osung.
\newblock {\em Math. Ann.}, 169:55--72, 1967.

\bibitem[San03]{sanders-spectral}
Jan~A. Sanders.
\newblock Normal form theory and spectral sequences.
\newblock {\em J. Differential Equations}, 192(2):536--552, 2003.

\bibitem[Ser87]{serre-semi}
J.-P. Serre.
\newblock {\em Complex Semisimple Lie Algebras}.
\newblock Springer-Verlag, 1987.

\bibitem[Sie42]{Siegel}
C.L. Siegel.
\newblock {Iterations of analytic functions}.
\newblock {\em Ann. Math.}, 43(1942)807-812, 1942.

\bibitem[SM71]{siegel-moser-book}
C.~L. Siegel and J.~K. Moser.
\newblock {\em Lectures on celestial mechanics}.
\newblock Springer-Verlag, New York-Heidelberg, 1971.
\newblock Translation by Charles I. Kalme, Die Grundlehren der mathematischen
  Wissenschaften, Band 187.

\bibitem[Sto00]{Stolo-ihes}
L.~Stolovitch.
\newblock Singular complete integrabilty.
\newblock {\em Publ. Math. I.H.E.S.}, 91:p.133--210, 2000.

\bibitem[Sto05]{stolo-annals}
L.~Stolovitch.
\newblock Normalisation holomorphe d'alg\`ebres de type {C}artan de champs de
  vecteurs holomorphes singuliers.
\newblock {\em Ann. of Math. (2)}, 161(2):589--612, 2005.

\bibitem[S{\.Z}02]{zoladek-nil}
E.~Str{\'o}{\.z}yna and H.~{\.Z}o{\l}adek.
\newblock The analytic and formal normal form for the nilpotent singularity.
\newblock {\em J. Differential Equations}, 179(2):479--537, 2002.

\bibitem[S{\.Z}08]{zoladek-nil-multidim}
E.~Str{\'o}{\.z}yna and H.~{\.Z}o{\l}adek.
\newblock Multidimensional formal {T}akens normal form.
\newblock {\em Bull. Belg. Math. Soc. Simon Stevin}, 15(5, Dynamics in
  perturbations):927--934, 2008.

\bibitem[S{\.Z}11]{zoladek-nil-div}
E.~Str{\'o}{\.z}yna and H.~{\.Z}o{\l}adek.
\newblock Divergence of the reduction to the multidimensional nilpotent
  {T}akens normal form.
\newblock {\em Nonlinearity}, 24(11):3129--3141, 2011.

\bibitem[Tak73]{kn:Tak2}
F.~Takens.
\newblock {Singularities of vector fields}.
\newblock {\em I.H.E.S}, 43(1973),47-100, 1973.

\bibitem[Vey78]{vey-ham}
J.~Vey.
\newblock Sur certains syst\`{e}mes dynamiques s\'{e}parables.
\newblock {\em Am. Journal of Math. 100}, pages 591--614, 1978.

\bibitem[Vey79]{vey-iso}
J.~Vey.
\newblock Alg\`{e}bres commutatives de champs de vecteurs isochores.
\newblock {\em Bull. Soc. Math. France,107}, pages 423--432, 1979.

\bibitem[Wei32]{weitzenbock}
R.~Weitzenb{\"o}ck.
\newblock \"{U}ber die {I}nvarianten von linearen {G}ruppen.
\newblock {\em Acta Math.}, 58(1):231--293, 1932.

\bibitem[Zun02]{zung-nf}
Nguyen~Tien Zung.
\newblock Convergence versus integrability in {P}oincar\'e-{D}ulac normal form.
\newblock {\em Math. Res. Lett.}, 9(2-3):217--228, 2002.

\bibitem[Zun05]{zung-birkhoff}
Nguyen~Tien Zung.
\newblock Convergence versus integrability in {B}irkhoff normal form.
\newblock {\em Ann. of Math. (2)}, 161(1):141--156, 2005.

\end{thebibliography}
